\documentclass[12pt,a4paper,leqno,final]{article}

\usepackage{amsfonts,delarray,amssymb,amsmath,amsthm}
\usepackage[notcite,notref]{showkeys}
\usepackage{epsfig,graphicx,color}

\newtheorem{theo}{Theorem}
\newtheorem{lem}{Lemma}[section]
\newtheorem{defi}{Definition}
\newtheorem{pro}{Proposition}[section]

\newtheorem{game}{Game}

\theoremstyle{remark}
\newtheorem{remark}{Remark}[section]

\def\vp{\varphi}
\def\({\left(}
\def\){\right)}
\def\l|{\left|}
\def\r|{\right|}

\def\ue{u^\ep}

\def\ep{\varepsilon}
\def\eps{\varepsilon}
\def\mr{\mathbb{R}}

\def\nab{\nabla}

\def\hal{\frac{1}{2}}

\def\a{\alpha}

\def\p{\partial}

\def\uB{\overline{u}}
\def\ub{\underline{u}}

\def\k{\kappa}
\def\G{\Gamma}
\def\v{v}
\def\vp{\varphi}

\def\RE{\mathcal{R}^\ep}
\def\Re{\mathcal{R}_\ep}
\def\Sc{\mathcal{S}}

\def\dt0{{{\frac{d}{dt}}_{|t=0}}}
\def\d2t0{{{\frac{d^2}{dt^2}}_{|t=0}}}
\numberwithin{equation}{section}

\def\un{\mathbf{1}}
\def\R{\mathbb{R}}

\begin{document}

\title{Repeated games for eikonal equations, integral curvature
 flows and non-linear parabolic integro-differential equations}

\author{ Cyril Imbert\footnote{Universit\'e Paris-Dauphine, CEREMADE,
    place de Lattre de Tassigny, 75775 Paris cedex 16, France, \texttt{imbert@ceremade.dauphine.fr}}~
  and Sylvia Serfaty\footnote{Universit\'e Pierre et Marie Curie
(Paris 6), laboratoire Jacques-Louis Lions \& Courant Institute of Mathematical Sciences, New York University, \texttt{serfaty@ann.jussieu.fr}}
}

\date{\today}

\maketitle

\begin{abstract}
  The main purpose of this paper is to approximate several non-local
  evolution equations by zero-sum repeated games in the spirit of the
  previous works of Kohn and the second author (2006 and 2009):
  general fully non-linear parabolic integro-differential equations on
  the one hand, and the integral curvature flow of an interface
  (Imbert, 2008) on the other hand. In order to do so, we start by
  constructing  such a game for eikonal equations whose speed has a
  non-constant sign. This provides a (discrete) deterministic control
  interpretation of these evolution equations.

  In all our games, two players choose positions successively, and
  their final payoff is determined by their positions and additional
  parameters of choice.  Because of the non-locality of the problems
  approximated, by contrast with local problems, their choices have to
  ``collect'' information far from their current position. For
  integral curvature flows, players choose hypersurfaces in the whole
  space and positions on these hypersurfaces.  For parabolic
  integro-differential equations, players choose smooth functions on
  the whole space.
\end{abstract}

 \medskip


\section{General introduction}

Kohn and the second author gave in \cite{ks0} a deterministic control
interpretation for motion by mean curvature and some other geometric
laws. In particular, given an initial set $\Omega_0 \subset \mr^N$,
they prove that the repeated game invented by Joel Spencer (originally
called ``pusher-chooser'' game, now sometimes known as the ``Paul-Carol''
game) \cite{spencer} converges towards the
mean curvature motion of $\p\Omega_0$. In a second paper \cite{ks},
they construct analogous approximations of general fully non-linear
parabolic and elliptic equations.

This paper is concerned with extending this approach to several
non-local evolutions. In particular we construct a zero-sum repeated
game with two players for a geometric motion originally introduced to
describe dislocation dynamics \cite{ahlm}. This motion also appears in
\cite{cafsoug} where threshold dynamics associated with kernels
decaying slowly at infinity are considered. It was recently
reformulated by the first author \cite{imbert} in order to deal with
singular interacting potentials. Such a motion is referred to as the
integral curvature flow; it also appears in
\cite{is,cafroqsav}. See the introduction of \cite{imbert} for
more details.

In order to construct such a game, we start with the simpler guiding
case of the eikonal equation associated with a changing sign velocity,
for which we give a game approximation.  We are guided by the ideas of
Evans and Souganidis \cite{evanssouganidis}; they proved in particular
that the solution of the eikonal equation can be represented by the
value function of a differential game. Our first task is thus to give
a discrete version of such a representation.

The specificity of the integral curvature flow is that it is non-local
in the sense that its normal speed at a boundary point $x$ not only
depends on the front close to $x$ (such as the outer normal unit
vector or the curvature tensor) but also on the whole curve. Indeed,
the integral curvature is a singular integral operator. This is the
reason why we also contruct a game to approximate general fully
non-linear parabolic equations involving singular integral terms.

The framework of viscosity solutions \cite{cl81,cil92} and the
level-set approach \cite{os,cgg,evansspruck} are used in order to define
properly the various geometric motions. We recall that the level-set
approach consists in representing the initial interface as the
$0$-level set of a (Lipschitz) continuous function $u_0$, looking for
the evolving interface under the same form, proving that the function
$u(t,x)$ solves a partial differential equation and finally proving
that the $0$-level set of the function $u(t,\cdot)$ only depends on
the $0$-level set of $u_0$. The proofs of convergence follow the
method of Barles-Souganidis \cite{bs91} i.e. use the stability,
monotonicity and consistency of the schemes provided by our games.

There are several motivations for constructing such games. First, it
shows that viscosity solutions of an even wider class of equations
have a deterministic control representation; while previously this was
known to be true only for first order Hamilton-Jacobi equations, and
then since \cite{ks0,ks} for general local second order PDE's.  Seen
differently, it shows that a wide class of non-local evolutions
 have a minimax formulation.  Then, these games can serve
to build robust numerical schemes to approximate the solutions to the
equations.  Finally, although this has not been achieved very much,
they could in principle serve to obtain new qualitative information on
the solutions to the PDE's.

The game we present for integral curvature flow, even though this is a
geometric evolution, is much more complicated that the Paul-Carol
game studied in \cite{ks0}. It would be very nice to find a game whose
rules are simpler and which would be a natural generalization of the
Paul-Carol game. However, we do not know at this stage whether this
is possible.

The paper is organized as follows. In Section 2, we present the
various equations that we study, state the definitions, present the
games and give the main convergence results: first for parabolic
integro-differential equations (in short PIDE), second for eikonal
equations, and third for integral curvature flows.  In Section 3, we
return to these theorems in order and give their proofs.

\paragraph{Notation.} The unit ball of $\mr^N$ is denoted by $B$.  A
ball of radius $r$ centered at $x$ is denoted by $B_r(x)$.  The
function $\un_A (z)$ is defined as follows: $\un_A (z) = 1$ if $z \in
A$ and $0$ if not.  The unit sphere of $\R^N$ is denoted by
$\mathbb{S}^{N-1}$. The set of symmetric real $N \times N$ matrices is
denoted by $S_N$.

Given two real numbers $a,b$, $a \wedge b$ denotes $\min (a,b)$ and $a
\vee b$ denotes $\max (a,b)$. Moreover, $a_+$ denotes $\max(0,a)$ and
$a_-= \max(0,-a)$.

The time derivative, space gradient and Hessian matrix
 of a function $\phi$ are respectively denoted by $\partial_t \phi$, $D \phi$
and $D^2 \phi$.

$C^2_b(\R^N)$ denotes the space of $C^2$ bounded functions
such that their first and second derivatives are also bounded.

\paragraph{Acknowledgements.} The first author was partially
supported by the ANR project MICA from the French Ministry of
Research, the second by an EURYI award.

\section{Main Results}

This section is devoted to the description of the games we introduce to
approximate the  various geometric motions or solutions of parabolic
PIDE.

Following \cite{ks0,ks}, in each game there are two opposing players
Paul and Carol (or sometimes Helen and Mark). Paul starts at point $x$
at time $t>0$ with zero score. At each step $n$, the position $x_n$
and time $t_n$ are updated by using a small parameter $\eps>0$:
$(t_n,x_n) = (t_n(\eps),x_n(\eps))$. The game continues until the
running time $t_N$ is larger than a given final time $T$. At the end
of the game, Paul's final score is $u_T(x_N)$ where $u_T$ is a given
continuous function $u_T$ defined on $\mr^N$, and $x_N$ is the final
position. Paul's objective is to maximize his final score and Carol's
is to obstruct him.

We define the value function $\ue$ of the game starting at $x$ at time
$t$ as
\begin{equation}
  \label{valf} \ue(t,x)= \max \(\text{final score for Paul starting
    from } (t,x)\) \, .
\end{equation}
The main results of this paper assert that the value functions
associated with the games described in the next subsections converge
to solutions of the corresponding evolution equations. As it is
natural for control problems, the framework to use is that of
viscosity solutions.

We present the games in increasing order of complexity, so   we start by presenting the results for
parabolic integro-differen\-tial equations (PIDE).

\subsection{General Parabolic Integro-Differential Equations}

The parabolic non-linear integro-differential equations at stake in
this paper are of the following form
\begin{equation}\label{eq:pide}
-\partial_t u + F(t,x,Du ,D^2 u, I [x,u]) = 0 \quad \mbox{ in } (0,T] \times \mr^N
\end{equation}
where $T>0$ is a final time, $F$ is a continuous non-linearity
satisfying a proper ellipticity condition (see below) and $I[x,U]$ is
a singular integral term defined for $U : \R^N \to \R$ as follows
\begin{equation}
  \label{defI} I [x,U]  = \int [ U(x+z) - U (x) - DU(x) \cdot z
  \un_B (z) ] \nu (dz)
\end{equation}
where we recall $B$ is the unit ball, $\un_B(z) = 1$ if $|z| < 1$ and
$0$ if not, and $\nu$ is a non-negative singular measure satisfying
\begin{equation}\label{cond:nuPIDE}
\int_B |z|^2 \nu (dz) < + \infty , \quad \int_{\mr^N \setminus B}
\nu (dz) < +\infty\, .
\end{equation}
We also assume for simplicity that $\nu (dz) = \nu (-dz)$ but this is
not a restriction. Such measures are referred to as (symmetric) L\'evy
measures and associated integral operators $I[x,U]$ as L\'evy
operators. Such equations appear in the context of mathematical
finance for models driven by jump processes; see for instance
\cite{ct}. Because of the games we construct, a \emph{terminal}
condition is associated with such a parabolic PIDE. Given a final time
$T>0$, the solution $u$ of \eqref{eq:pide} is submitted to the
additional condition
\begin{equation}\label{eq:tc}
u(T,x) = u_T (x)
\end{equation}
where $u_T:\R^N \to \R$ is the terminal datum.  The equation is called
parabolic when the following ellipticity condition is fulfilled
\begin{equation}\label{cond:ellip}
A \le B, l \le m \Rightarrow F(t,x,p,A,l) \ge F(t,x,p,B,m) \; ,
\end{equation}
where $A\le B$ is meant with respect to the order on symmetric
matrices.  Under this condition, the equation with terminal condition
(\ref{eq:tc}) is well-posed in $(0,T] \times \R^N$.  \medskip

\subsubsection{Viscosity solutions for PIDE}

In this section, we recall the definition and framework of viscosity
solutions for \eqref{eq:pide} \cite{soner,sayah}. Since we will work
with bounded viscosity solutions, we give a definition in this
framework.
\begin{defi}[Viscosity solutions for PIDE]\label{def:visc-pide}
Consider $u : (0,T) \times \mr^N \to \R$, a bounded function.
\begin{enumerate}
\item It is a \emph{viscosity sub-solution} of \eqref{eq:pide} if
it is upper semi-continuous and if for every  bounded
test-function $\phi \in C^2$  such that $u-\phi$ admits a global maximum
$0$ at $(t,x) \in (0,T) \times \R^N$, we have
\begin{equation}\label{subsol-pide}
 -\p_t \phi (t,x) +  F(t,x,D\phi (x),D^2 \phi (x), I [x,\phi])   \le 0 \; .
\end{equation}
\item It is a \emph{viscosity super-solution} of \eqref{eq:geom} if
it is lower semi-continuous and if for every  bounded
test-function $\phi \in C^2$ such that $u-\phi$ admits a  global minimum
$0$ at $(t,x) \in (0,T) \times \R^N$, we have
\begin{equation}\label{sursol-pide}
 -\p_t \phi (t,x) +  F(t,x,D\phi (x),D^2 \phi (x), I [x,\phi])   \ge 0 \; .
\end{equation}
\item A continuous function $u$ is a \emph{viscosity solution} of
\eqref{eq:pide} if it is both a sub and super-solution.
\end{enumerate}
\end{defi}
\begin{remark} \label{rem:supp-nu-compact} If the mesure $\nu$ is
  supported in $B_R$, then the global maximum/minimum $0$ of $u-\phi$ at
  $(t,x)$ can be replaced with a strict maximum/minimum $0$ in $(0,T)
  \times B_{R'}(x)$ for any $R' \ge R$. Indeed, changing $\phi$ outside
$B_R(x)$ does not change the value of $I[x,\phi]$ in this case.
\end{remark}
On the one hand, in order for the value of the repeated game we are
going to construct to be finite, we need
to make  some growth assumption on the nonlinearity $F$.
On the other hand, in order to get the convergence of the value of the
repeated game, the comparison principle for
\eqref{eq:pide} has to hold. For these reasons we assume that
$F$ satisfies the ellipticity condition given above together with
the following set of assumptions (see \cite{barlesimbert}):
\medskip

\noindent \textsc{Assumptions (A).}
\begin{itemize}
\item (A0) $F$ is continuous on $\R\times \R^N \times \R^N \times S_N
  \times \mr$. \item (A1) There exist constants $k_1>0$, $k_2 >0$ and
  $C>0$ such that for all $(t,x,p,A) \in \R \times \R^N \times \R^N
  \times \R^N \times S_N, $ we have
$$
|F(t,x,p,A,0)| \le C (1 + |p|^{k_1} + |X|^{k_2}) \, .
$$
\item (A2-1) For all $R>0$, there exist moduli of continuity
$\omega, \omega_R$ such that, for all $|x|,|y|\le R$, $|v|\le R$,
$l\in\R$ and for all $X,Y \in S_N $ satisfying
\begin{equation}\label{ineqmat}
\left[\begin{array}{cc}X&0\\0&-Y\end{array}\right]
 \le \frac1\eps \left[\begin{array}{cc}I&-I\\-I&I\end{array}\right]
+ r(\beta) \left[\begin{array}{cc}I&0\\0&I\end{array}\right]
\end{equation}
for some $\eps >0$ and $r(\beta) \to 0$ as $\beta \to 0$ (in the sense of matrices in $S_{2N}$), then, if
$s(\beta) \to 0$ as $\beta \to 0$, we have
\begin{multline}\label{cond:a31}
F(t,y,v,\eps^{-1} (x-y), Y,l) - F(t,x,v,\eps^{-1} (x-y) + s(\beta),X,l) \\
\le \omega (\beta)+\omega_R (|x-y|+\eps^{-1}|x-y|^2)
\end{multline}
\noindent  or \item (A2-2) For all $R>0$, $F$ is uniformly
continuous on $[-R,R] \times \R^n \times B_R \times D_R \times \R$
where $D_R := \{ X \in S_N;\ |X| \leq R\}$ and there exist a
modulus of continuity $\omega_R$ such that, for all $x,y \in
\R^N$, $|v|\le R$, $ l\in\R$ and for all $X,Y \in S_N$ satisfying
\eqref{ineqmat} and $\eps >0$, we have
\begin{equation}\label{cond:a32}
F(t,y,v,\eps^{-1} (x-y), Y,l) - F(t,x,v,\eps^{-1} (x-y),X,l) \le
\omega_R (|x-y|+ \eps^{-1} |x-y|^2 )\; .
\end{equation}
\item (A3) $F(t,x,u,p,X,l)$ is Lipschitz continuous in $l$,
uniformly with respect to all the other variables.
\end{itemize}
Assumptions (A0)-(A1)   are all we need to show that the relaxed
semi-limits of our value functions are viscosity sub- (resp.
super-)solutions to (\ref{eq:pide}). Assumptions (A2)-(A3) are
meant to ensure that a comparison principle holds for (\ref{eq:pide}), \textit{i.e.}
that viscosity sub-solutions are smaller than viscosity
super-solutions, which guarantees the final convergence.

\subsubsection{The game for PIDE}

We are given positive parameters $\ep,R>0$. A truncated integral
operator $I_R [x,\Phi]$ is defined by replacing in \eqref{defI}
$\nu(dz)$ with $\un_{B_R} (z) \nu (dz)$. We also consider a
positive real number $\alpha \in (0, (\max(1,k_1,k_2))^{-1})$ where
the constants $k_1,k_2$ appear in Assumption~(A1). In
this setting, for the sake of consistency with \cite{ks} where a
financial interpretation was given, the players should be Helen
(standing for hedger) and Mark (standing for market), with Helen
trying to maximize her final score under the opposition of Mark.
\begin{game}[Parabolic PIDE]
At time $t\in (0, T)$, the game  starts at $x$ and Helen has a zero
score. Her objective is to get the highest final score.
\begin{enumerate}
\item Helen  chooses a function $\Phi \in C^2_b(\mr^N)$ such that
$\|\Phi\|_\infty \le \eps^{-\alpha}$, $|D \Phi (x)| \le \eps^{-\alpha}$ and $|D^2 \Phi (x)| \le
\eps^{-\alpha}$. \item Mark chooses  the new position $y\in B_R
(x)$. \item Helen's score is increased by
$$
\Phi(x)- \Phi(y) - \ep F(t,x, D\Phi(x), D^2 \Phi(x),I_R[x, \Phi])
\, .
$$
Time is reset to $t+ \ep$. Then we repeat the previous steps until
time is larger than $T$. At that time, Helen collects the bonus
$u_T(x)$, where $x$ is the current position of the game.
\end{enumerate}
\end{game}

\subsubsection{Theorem and comments}

It is possible to construct a repeated game that approximates a PIDE
where $F$ also depends on $u$ itself, but its formulation is a bit
more complicated. This is important from the point of view of
applications but since, with the previous game at hand, ideas from
\cite{ks} can be applied readily, we prefer to present it in this
simpler framework.

The dynamic programming principle is, in this case,
\begin{multline}
  \label{pdp:pide} \ue(t,x) = \sup_{\stackrel{\Phi \in C^2(\mr^N)}{\|\Phi\|_\infty, |D
      \Phi (x)|, |D^2 \Phi (x)|\le \eps^{-\alpha} } }
  \inf_{y \in B_R (x)} \bigg\{ \ue(t+ \ep,y) \\
  + \Phi(x)- \Phi(y) - \ep F( t,x,D\Phi(x), D^2 \Phi(x), I_R[x,\Phi])
  \bigg\} \, .
\end{multline}

We will give below an easy formal argument that
allows to predict the following convergence result.
\begin{theo} \label{theo:conv-pide} Assume that $F$ is elliptic and
  satisfies {\rm (A0)} and {\rm (A1)}. Assume also that $u_T \in
  W^{2,\infty}(\R^N)$. Then the upper (resp. lower) relaxed
  semi-limit $\uB$ (resp. $\ub$) of $(\ue)_{\eps >0}$ is a
  sub-solution (resp. super-solution) of \eqref{eq:pide} and
$$
\uB (T,x) \le u_T (x) \le \ub (T,x) \,.
$$
In particular, if $F$ also satisfies {\rm (A2)}, {\rm (A3)}, then $\ue
$ converges locally uniformly in $\R \times \R^N$ towards
the viscosity solution $u$ of \eqref{eq:pide}, \eqref{eq:tc} as $\eps
\to 0$ and $R\to +\infty$ successively.
\end{theo}
\begin{remark}
  As we mentioned, the second statement follows from the fact that
  (A2)--(A3) together with (A0) imply that the comparison principle
  for \eqref{eq:pide} holds true in the class of bounded functions.
\end{remark}
\begin{remark}\label{rem:stronger-result}
  We are in fact going to prove that under the same assumptions, $\ue$
  converges locally uniformly in $\R \times \R^N$ as $\eps \to 0$
  towards the viscosity solution of \eqref{eq:pide}, \eqref{eq:tc}
  where $I$ is replaced with the truncated integral operator.
  Theorem~\ref{theo:conv-pide} is then a direct consequence of this
  fact by using stability results such as the ones proved in
  \cite{barlesimbert}.
\end{remark}
\begin{remark}
 We assume that $u_T$ lies in $W^{2,\infty}(\R^N)$ for simplicity but
one can consider terminal data that are much less regular, for
instance bounded and uniformly continuous. However, this implies
further technicalities that we prefer to avoid here.
\end{remark}
\begin{proof}[Formal argument for Theorem~\ref{theo:conv-pide}]
  We assume that $u^\eps$ is smooth.  It is enough to understand why
  the following equality holds true
\begin{equation}\label{eq:formal}
u^\eps (t,x) =  u^\eps (t+\eps, x) - \eps F (t,x, D u^\eps (t+\eps,x), D^2 u^\eps (t+\eps, x),
I [x, u^\eps (t+\eps, \cdot)]) +o(\eps) .
\end{equation}
Indeed, after rearranging terms, dividing by $\eps$ and  passing to the limit, we get
$$
 - \partial_t u (t,x) + F(t, x, Du (t,x) ,D^2 u (t,x), I [x,u(t,\cdot)]) = 0 \, .
$$
It is easy to  see that if Helen chooses $\Phi = \ue(t+\ep, \cdot)$, Mark cannot
  change the score by acting on $y$. Indeed, the dynamic programming
  principle implies that $u^\eps(t,x)$ is larger than the right-hand side of
\eqref{eq:formal}.

It turns out that it is optimal for Helen to choose $\Phi
=\ue(t+\eps,\cdot)$.  In other words, the converse inequality holds
true (and thus \eqref{eq:formal} holds true too). To see this,
the dynamic programming principle tells us that it is enough to prove
that, for $\Phi \in C^2 (\R^N)$ fixed (with proper bounds), we have
\begin{multline*}
  \inf_{y \in B(x,R)} \{ u^\eps (t+\eps,y) + \Phi (x) - \Phi(y) - \eps
  F(t,x,D\Phi(x),D^2 \Phi(x), I[x,\Phi]) \} \\\le u^\eps (t+\eps, x) -
  \eps F (t,x, D u^\eps (t+\eps,x), D^2 u^\eps (t+\eps, x), I [x,
  u^\eps (t+\eps, \cdot)])+o(\eps) .
\end{multline*}
The following crucial lemma permits to conclude.  We recall that we assume that the
singular measure is supported in $B(0,R)$ for some $R>0$.
\begin{lem}[Crucial lemma for PIDE]\label{lf1}
  Let $F$ be continuous and $\Phi, \psi\in C^2(\mr^N)$ be two bounded
  functions.  Let $K$ be a compact subset of $\mr^N$ and let $x \in
  K$.  For all $\eps >0$, there exists $y=y_\eps \in B_R(x)$ such
  that
\begin{multline}\label{lf2}
\psi(y) + \Phi(x)- \Phi(y)
- \ep F(t,x,D\Phi(x),D^2 \Phi(x),I_R[x,\Phi]) \\
\le  \psi(x) - \eps F(t,x,D\psi(x), D^2 \psi(x), I_R[x,\psi]) +o(\ep)
\end{multline}
where the $o(\eps)$ depends on $F, \psi,\Phi$ and $K$ but not on $t,x,y$.
\end{lem}
The rigourous proof of this lemma is postponed until
Subsection~\ref{subsec:conv-pide}. However, we can motivate
this result by giving a (formal) sketch of its proof. Assume that the
conclusion of the lemma is false. Then there exists $\eta >0$ and
we have for all $y \in K$
$$
\psi (y)- \psi (x) > \Phi (y) - \Phi (x) + \eps (F(\dots) - F(\dots)) + \eta \eps  .
$$
In particular, $\psi (y)- \psi (x) > \Phi (y) - \Phi (x) + O(\eps)$.
This implies (at least formally)
\begin{eqnarray*}
D \psi (x) &= D \Phi (x) + o (1) \\
D^2 \psi (x) &\le D^2 \Phi (x) + o (1) \\
I [x,\psi] &\le I [x,\Phi] + o (1) .
\end{eqnarray*}
Then the ellipticity of $F$ implies that $F(\dots) - F(\dots) \ge o(1)$ and
we get the following contradiction: $0 \ge \eta o(1) + \eta \eps$.
\end{proof}

One can observe that this very simple game is a natural generalization
of the game constructed in \cite{ks} for fully non-linear parabolic
equations. Indeed, if $F$ does not depend on $I[\Phi]$, then all is
needed is proxies for $D\Phi(x), D^2\Phi(x)$. So instead of choosing a
whole function $\Phi$, Helen only needs to choose a vector $p$ (proxy
for $D\Phi(x)$) and a symmetric matrix $\Gamma$ (proxy for
$D^2\Phi(x)$), and replace $\Phi(y)- \Phi(x) $ in the score updating
by its quadratic approximation $$ p \cdot (y-x) + \hal\langle \Gamma
(y-x), (y-x)\rangle.$$ One then recovers the game of \cite{ks} (except
there $y$ is constrained to $B_{\ep^{1-\a}} (x)$). Of course it is
natural that for a non-local equation, local information at $x$ does
not suffice and information in the whole space needs to be collected
at each step.

\subsection{Level-set approach to geometric  motions}

Before stating our results for the geometric flows (eikonal equations
and integral curvature flow), we recall the level set framework for
such geometric evolutions.

The level-set approach \cite{os,cgg,evansspruck} consists in defining
properly motions of interfaces associated with geometric laws.  More
precisely, given an initial interface $\Gamma_0$, \textit{i.e.} the
boundary of a bounded open set $\Omega_0$, their time evolutions
$\{\Gamma_t\}_{t>0}$ and $\{\Omega_t\}_{t>0}$ are defined by
prescribing the velocity $V$ of $\Omega_t$ at $x \in \Gamma_t$ along
its normal direction $n(x)$ as a function of time $t$, position $x$,
normal direction $n(x)$, curvature tensor $Dn(x)$, or even the whole
set $\Omega_t$ at time $t$. The geometrical law thus writes
\begin{equation}\label{eq:geom-law}
V = G(t,x,n(x), Dn(x) , \Omega_t) \; .
\end{equation}
The level-set approach consists in describing $\Gamma_0$ and
$\{\Gamma_t\}_{t>0}$ as zero-level sets of continous functions $u_0$
(such as the signed distance function to $\Gamma_0$) and $u(t,\cdot)$
respectively
\begin{eqnarray*}
\Gamma_0 = \{ x \in \mr^N : u_0 (x) =0 \}
\quad & \mbox{ and } &\quad
\Omega_0 = \{ x \in \mr^N : u_0 (x) >0 \} \\
\Gamma_t = \{ x \in \mr^N : u (t,x) =0 \}
\quad & \mbox{ and }  & \quad
\Omega_t = \{ x \in \mr^N : u (t,x) >0 \} \; .
\end{eqnarray*}
The geometric law~\eqref{eq:geom-law}  translates into a fully non-linear
parabolic equation for $u$:
\begin{equation}\label{eq:geom}
  \partial_t u  =  G (t,x, \widehat{Du}, (I - \widehat{Du}\otimes \widehat{Du}) D^2 u, \Omega_t) |Du|
  := - F(t,x,Du, D^2 u, \Omega_t)
\end{equation}
(where $\hat{p}= |p|^{-1} p$ for $p \in \mr^N$, $p \neq 0$)
supplemented with the initial condition $u(0,x) = u_0 (x)$.  If proper
assumptions are made on the nonlinearity $F$, the level-set approach
is \emph{consistent} in the sense that, for two different initial
conditions $u_0$ and $v_0$ with the same $0$-level set, the associated
(viscosity) solutions $u$ and $v$ have the same zero-level sets at all
times as well. The interested reader is referred to
\cite{os,cgg,evansspruck} for fundamental results, \cite{bss} for
extensions and \cite{souganidis} for a survey paper.  \medskip

In the present paper, we deal with terminal conditions instead of
initial conditions. This is the reason why, for a given terminal time
$T>0$, we consider the equation $-\partial_t u + F = 0$ supplemented
with the terminal condition~\eqref{eq:tc}.  We will consider two
special cases of \eqref{eq:geom}
\begin{itemize}
\item the eikonal equation
\begin{equation}\label{eq:eikonal}
-\partial_t u - \v(x) |Du| =0
\end{equation}
\item and the integral curvature equation
\begin{equation}\label{eq:dislo}
-\partial_t u - \kappa [x,u] |Du | =0
\end{equation}
where $\kappa [x,u]$ is the integral curvature of $u$ at $x$
(see below for a definition).
\end{itemize}

\subsection{Eikonal equation}

The first geometric law \eqref{eq:geom-law} we are interested in is
the simple case where $V = \v(x)$ and
$$
\v : \R^N \to \R \text{ is a Lipschitz continuous function}
$$
and we do not assume that it has a constant sign. In this case, the
geometric equation~\eqref{eq:geom} reduces to the standard eikonal
equation \eqref{eq:eikonal}.

The solution of an eikonal equation can be represented as the value
function of a deterministic control problem when $\v$ has a constant
sign \cite{lions82}. If $v$ changes sign, it can be represented as the
value function of a deterministic differential game problem,
\textit{i.e.}, loosely speaking, a control problem with two opposing
players \cite{evanssouganidis}.

We recall the definition of a viscosity solution to the eikonal
equation
(\ref{eq:eikonal}).
\begin{defi}[Viscosity solution for (\ref{eq:eikonal})] \label{defi2}
  Given a function $u: (0,T) \times \mr^N\to \mr$, we say that
\begin{enumerate}
\item It is a \emph{viscosity sub-solution} of (\ref{eq:eikonal}) if it is
  upper semi-continuous and if for every test-function $\phi \in C^2$
  such that $u-\phi$ admits a local maximum  at $(t,x)\in
  (0,T)\times \mr^N$, we have
    \begin{equation}
    -\p_t \phi(t,x) - v(x) |\nab \phi|(t,x) \le 0.\end{equation}
\item It is a \emph{viscosity super-solution} of (\ref{eq:eikonal}) if it is
  lower semi-continuous and if for every test-function $\phi \in C^2$
  such that $u-\phi$ admits a local minimum  at $(t,x) \in (0,T)
  \times \mr^N$, we have
  \begin{equation}
    -\p_t \phi(t,x) - v(x) |\nab\phi|(t,x) \ge 0.\end{equation}
\item It is a viscosity solution of (\ref{eq:eikonal}) if it is both a
  sub and super-solution.
\end{enumerate}
\end{defi}
Here, we construct a semi-discrete approximation by discretizing time.
We use two opposing players Paul and Carol: Paul can take advantage of
$\v \ge 0$ to move, while when $\v \le 0$ it is Carol who takes
advantage to move (in the opposite direction).  We recall that $(\cdot)_+$
denotes the positive part and $(\cdot )_-$ the negative part of a
quantity.

\subsubsection{The game for the eikonal equation}

In order to describe the game, we introduce the following cut-off
function: for $\eps >0$ and $r>0$, we define
\begin{equation}\label{cep}
C_\eps (r) = (r \vee \eps^{\frac32})\wedge \eps^{\frac12} =
\left \{ \begin{array}{ll} \eps^{\frac32} & \text{if } 0 < r < \eps^{\frac32} , \\
r & \text{if } \eps^{\frac32} < r < \eps^{\frac12} , \\ \eps^{\frac12} &
\text{if } r > \ep^{\frac12} . \end{array} \right.
\end{equation}
This function is non-decreasing and  for every $r$ we have
$\eps^{\frac32} \le C_\eps(r) \le \eps^{\frac12}$.
\begin{game}[Eikonal equation] \label{game:eik} At time $t \in (0,T)$,
  Paul starts at $x$ with zero score. His objective is to get the
  highest final score.
\begin{enumerate}
\item Either $B_\ep(x)\cap \{ \v > 0\}\neq  \varnothing$, then
Paul chooses a point  ${x}_P \in B_\ep(x)\cap \{ \v> 0\}$ and time
gets reset to $t_P = t + C_\eps [ \eps (\v_+(x_P))^{-1} ]$.
\\ Or
$B_\ep(x)\cap \{ \v > 0\}=\varnothing$, then
 Paul stays at $x_P=x$  and time gets reset
to $t_P = t+ \eps^{2}$.

\item Either $B_\ep(x_P)\cap \{ \v< 0\}\neq  \varnothing$, then
Carol chooses a point  ${x}_C \in B_\ep(x_P)\cap \{ \v< 0\}$ and
time gets reset to $t_C = t_P + C_\eps [ \eps (\v_-(x_C))^{-1} ]$.
\\ Or
$B_\ep(x_P)\cap \{ \v< 0\}= \varnothing$, then
 Paul stays at $x_C=x_P$  and time gets reset
to $t_C = t_P+ \eps^{2}$. \item The players repeat the two previous steps
until $t_C \ge T$. Paul's final score is $u_T (x_C)$ where $x_C$ s the final position of the game.
\end{enumerate}
\end{game}
\subsubsection{Result and remarks}

The previous game can be translated as follows: let  for short $E^+$
and $E^-$ denote the sets
\begin{equation}\label{EP} E^\pm(x)=
  \begin{cases} B_\ep(x) \cap\{ \pm v>0\} & \text{ if }  B_\ep(x)
    \cap \{ \pm v>0\}
    \neq \varnothing\\
    \{x\}  & \ \text{if not.}
\end{cases}
\end{equation}
   With this notation
\begin{equation}\label{pdp:eikonal}
  \ue(t,x)=  \sup_{ x_P\in   E^+(x)}
  \left\{ \inf_{ x_C \in  E^-(x_P) }
    \left\{ \ue (t_C,x_C) \right\} \right\},
\end{equation}
where
\begin{equation}\label{eq:timereset}
  \left\{\begin{array}{l} t_P =  t + \left\{ \begin{array}{ll} C_\ep
          [ \eps (\v_+(x_P))^{-1} ] & \mbox{ if } B_\ep(x) \cap\{ \v>0\}
          \neq \varnothing
          \\
          \eps^{2} & \mbox{ if not }
        \end{array}\right.  \\
      t_C = t_P + \left\{ \begin{array}{ll}
          C_\ep[ \eps (\v_-(x_C))^{-1} ]  & \mbox{ if } B_\ep(x_P) \cap \{ v<0\}\neq \varnothing\\
          \eps^{2} & \mbox{ if not}
        \end{array}\right.
\end{array}\right.
\end{equation}
and
$$
\ue (t,x) = u_T (x) \text{ if } t \ge T \; .
$$

We will refer to \eqref{pdp:eikonal} as the dynamic programming
principle for Game~\ref{game:eik} even if only one time step is
considered. We next claim that the following convergence result holds
true; again the limiting equation can be predicted by a formal
argument from (\ref{pdp:eikonal}) (see below).
\begin{theo} \label{theo:conv-eikonal} Assume that $\v$ is
  Lipschitz continuous and $u_T$ is bounded and Lipschitz
  continuous. Then the function $\ue$ converges locally uniformly as
  $\ep \to 0$ towards the unique viscosity solution of
  \eqref{eq:eikonal}, \eqref{eq:tc}.
\end{theo}
\begin{remark}
  Let us mention that the parameters $\alpha= \frac12$ and $\beta=\frac32$ in
  the definition of $C_\ep$ \eqref{cep} really only need to satisfy $1< \alpha < \beta <
  2$.
\end{remark}
We next give the formal argument which permits to predict the
convergence result for the eikonal equation.
\begin{proof}[Formal argument for Theorem~\ref{theo:conv-eikonal}]
  We first rewrite the dynamic programming principle as follows
\begin{equation*}
  0=\sup_{ x_P \in E^+ (x) } \bigg\{\ue
  (t_P,x_P) - \ue (t,x)  + \inf_{ x_C \in E^-(x_P) } \left\{ \ue (t_C,x_C) -\ue
    (t_P,x_P)\right\}\bigg\} \, .
\end{equation*}

We only treat the case $v(x) >0$ because  the argument is completely
analogous in the case $v(x)<0$. Hence, for $\eps$ small enough,
$B_\eps(x) \subset \{ v>0\}$, $B_\eps(x_P) \cap \{ v<0 \}
=\varnothing$ and $(t_C,x_C)=(t_P+\eps^2,x_P)$. The previous
equality then yields (approximating $C_\ep(r) $ by $r$)
\begin{eqnarray*}
0 &=& \sup_{ x_P \in  B_\ep(x) } \bigg\{\ue (t_P,x_P) - \ue (t,x)  + O(\ep^2) \bigg\} \\
&=& \sup_{x_P \in B_\eps (x)} \bigg\{ \partial_t \ue (t,x) (t_P -t) + D \ue (t,x) (x_P-x) \bigg\} + o (\eps)\\
&=& \frac{\eps}{v(x_P)} ( \partial_t \ue (t,x) + v(x_P) |D \ue (t,x)| ) + o(\eps) \\
&=&\frac{\eps}{v(x_P)} ( \partial_t \ue (t,x) + v(x) |D \ue (t,x)| ) + o(\eps) \, .
\end{eqnarray*}
Hence,  dividing by $\eps/v(x_P)$ and  letting $\eps \to 0$, we obtain
formally
$$
\partial_t u (t,x) + v(x) |D u | (t,x) =0 \, .
$$
\end{proof}

\subsection{Integral curvature flow}

\subsubsection{Definitions}

Even if the authors do not use this word, the notion of integral curvature
is considered in papers such as
\cite{dfm,fim,cafsoug,imbert,cafroqsav}. Here is the definition we will take.

Consider a function $K: \R^N \to (0,+\infty)$ such that
\begin{equation}
\label{cond:N}
\left\{ \begin{array}{l}
K \text{ is even, supported in } B_R(0) \\
 K\in W^{1,1} (\R^N \setminus B_\delta(0)) \text{ for all } \delta >0 \\
\int_{B_\delta(0)} K = o \left( \frac1\delta \right) \\
\int_{\mathcal{Q}(r,e)} K < +\infty \text{ for all } r>0, e \in \mathbb{S}^{N-1} \\
\int_{\mathcal{Q}(r,e)} K = o \left( \frac1{r} \right)
\end{array} \right.
\end{equation}
where $\mathcal{Q} (r,e)$ is a paraboloid defined as follows
$$
\mathcal{Q} (r,e) = \{ z \in \R^N : r |z \cdot e| \le |z-(z\cdot e)e |^2 \} .
$$
Interesting examples of such $K$'s include
$$
K(z) = \frac{C(z)}{1 + |z|^{N+\alpha}} \quad
\text{ or } \quad  K(z) =\frac{C(z)}{|z|^{N+\alpha}}
$$
for some  cut-off function $C:\R^N \to \R$ which is even, smooth and supported in $B_R(0)$.

\begin{remark}
  It is not necessary to assume that $K$ has a compact support in
  order to define the non-local geometric flow. However, we need this
  assumption in order to construct the game and prove that it
  approximates the geometric flow. We can then later follow what we
  did when dealing with PIDE: approximate any integral curvature flow
  by first approximating $K$ by kernels $K^R$ äcompactly supported in
  $B_R(0)$ and by taking next the limit of the corresponding value
  functions as $\ep \to 0$ and $R \to \infty$ respectively.
\end{remark}

Consider $U \in C^2$ such that $DU(x) \neq 0$.
We define
\begin{eqnarray*}
\k^* [x,U] = K * \un_{\{ U \ge U (x) \}} -  K * \un_{\{ U < U (x) \}} \\
\k_* [x,U] = K* \un_{\{ U > U (x) \}} -  K * \un_{\{ U \le U (x) \}}.
\end{eqnarray*} These functions coincide if for instance $DU \neq 0$
on $\{U = U(x)\}$. They define the integral curvature of the
``hypersurface'' $\{ U(z) = U(x) \}$ at the point $x$.  The reader can
notice that this ``hypersurface'' is oriented \textit{via} the sign of
the function $U$. The classical curvature can be recovered if
$K(z) = \frac{1-\alpha}{|z|^{N+\alpha}}$ and $\alpha \to 1$, $\alpha <1$; see \cite{imbert}.

Functions $\k^*$ and $\k_*$ enjoy the following properties (see \cite{imbert}):
\begin{enumerate}
\item  {\sc Semi-continuity:} functions $\k^* [\cdot,U]$ and
$\k_*[\cdot,U]$ are respectively upper and lower
semi-continuous
\begin{eqnarray}
\label{eq:usc-courbure}
\k^* [x,U] &\ge& \limsup_{y \to x} \k^*[y,U]; \\
\label{eq:lsc-courbure}
\k_* [x,U] &\le& \liminf_{y \to x} \k^*[y,U];
\end{eqnarray}
\item {\sc Monotonicity property:} 
\begin{eqnarray}
\label{eq:monotonie++}
\{ U \ge U (x) \} \subset \{ V \ge V(x)\} &\Rightarrow&
\k^*[x,U] \le \k^*[x,V], \\
\nonumber
\{ U > U (x) \} \subset \{ V > V(x)\} &\Rightarrow&
\k_*[x,U] \le \k_*[x,V].
\end{eqnarray}
\end{enumerate}
We next make precise the notion of viscosity solutions for
\eqref{eq:dislo}, \eqref{eq:tc} that will be used in the present paper.
\begin{defi}[Viscosity solutions for \eqref{eq:dislo}]\label{def:visc-geom}
Given a function $u : (0,T) \times \mr^N \to \R$, we say that
\begin{enumerate}
\item It is a \emph{sub-solution} of \eqref{eq:dislo} if it is upper
  semi-continuous and if for every test-function $\phi \in C^2$ such
  that $u-\phi$ admits a strict  maximum in $(0,T) \times B_{R +1}
  (x)$ at $(t,x)$, we have
\begin{equation}\label{subsol}
- \p_t \phi (t,x) -  \k^*[x,\phi(t,\cdot)]  |D \phi|(t,x)\   \le 0
\end{equation}
if $D \phi (t,x) \neq 0$ and $-\p_t \phi(t,x) \le 0$ if $D\phi(t,x)=0$;
\item It is a \emph{super-solution} of \eqref{eq:dislo} if it is lower
  semi-continuous and if for every test-function $\phi \in C^2$ such
  that  $u-\phi$ admits a strict
   minimum in $(0,T) \times B_{R +1} (x)$ at $(t,x)$, we have
\begin{equation}\label{sursol}
-\p_t \phi (t,x) -  \k_*[x,\phi(t,\cdot)]  |D \phi|(t,x)  \ge 0
\end{equation}
if $D \phi (t,x) \neq 0$ and $-\p_t \phi(t,x) \ge 0$ if $D\phi(t,x)=0$;
\item
It is a \emph{solution} of \eqref{eq:dislo} if it is both a sub and super-solution.
\end{enumerate}
\end{defi}
It is proved in \cite{imbert} that a comparison principle holds
true for such super- and sub-solutions.

\subsubsection{The game for the integral curvature equation}

In the following game, the two players choose successively
hypersurfaces and points on these hypersurfaces or close to them.  In
the present paper, a hypersurface refers to the $0$-level set of a
smooth function $\phi$. We  recall the definition of the cut-off
function we considered in the repeated game for the eikonal case.
\begin{equation}\label{cep2}
C_\eps (r) = (r \vee \eps^{\frac32})\wedge \eps^{\frac12} =
\left \{ \begin{array}{ll} \eps^{\frac32} & \text{if } 0 < r < \eps^{\frac32} , \\
    r & \text{if } \eps^{\frac32} < r < \eps^{\frac12} , \\
    \eps^{\frac12} & \text{if } r > \ep^{\frac12} . \end{array}
\right.
\end{equation}
We also recall that $R$ is the size of the support of $\kappa$ as in
\eqref{cond:N}.
\begin{game}[Integral curvature equation]
At time $t \in (0,T)$, Paul starts at $x$ with zero score. His objective is to get
the highest final score.
\begin{enumerate}
\item
  Paul chooses a point $x_P^+ \in B_\ep(x)$
   and a hypersurface $\G^+$ passing through $x_P^+$ defined by
  $$
  \G^+ = \{ z \in \R^N : \phi^+ (z ) = \phi^+ (x_P^+) \}$$
  with $\phi^+ \in C^2 (\R^N)$,
 oriented through the requirement   $\phi^+ (x)
   \le \phi^+(x_P^+)$.
  \begin{itemize}
  \item If $D \phi^+(x_P^+)\neq 0$ and $\k^*[x_P^+,\G^+]  >0$, Carol chooses the new position
    point $x_C^+$ in the half-space delimited   by $ \G^+$ \textit{i.e.} in
    $\{z \in B_R(x_P^+) :  \phi^+(z)  \ge \phi^+(x_P^+) \}$.
    Time gets reset to $t^+=t + C_\ep(\ep \k^*[x_P^+,\G^+]^{-1})$.
  \item If $D\phi^+(x_P^+)=0$ or $\k^* [x_P^+,\G^+] \le 0$, then the game stays
    at $x$: $x_C^+=x$. Time gets reset to
    $t^+=t + \ep^2$.
\end{itemize}
\item From the new position $x_C^+$ and time $t^+$ determined
  above, Carol chooses  a point $x_C^- \in B_\ep (x_C^+)$  and a
  hypersurface $\Gamma^-$ passing through $x_C^-$ defined by
  $$
  \G^- = \{ z \in \R^N : \phi^- (z ) = \phi^- (x_C^-) \}
  $$
  with $\phi^- \in C^2 (\R^N)$, and oriented through the requirement
 $\phi^-$ is such that $\phi^- (x_C^+) \ge \phi^-(x_C^-)$.
  \begin{itemize}
  \item If $D\phi^-(x_C^-)\neq 0$ and $\k_*[x_C^-,\G^-] < 0$, Paul chooses the
   new position point $x_P^- $ in the half-space delimited by $\G^-$
   \textit{i.e.} in $\{ z\in B_R(x_C^-) : \phi^-(z)\le \phi^-(x_C^-)\}$.
    Time gets reset to $t^-=t^+ + C_\ep (\ep | \k_*[x_C^-,\G]|^{-1})$.
  \item If $D\phi^-(x_C^-)= 0$ or $\k_* [x_C^-,\G^-] \ge 0$, then the game stays at  $x_C^+$ ($x_C^-=x_C^+$)
    and time gets reset to $t^- = t^+ + \ep^2$.
  \end{itemize}
\item
  Then previous steps are repeated as long as $t^- < T$. Paul's final score is $u_T (x_P^-)$.
\end{enumerate}
\end{game}
Remark in particular that in Step~1, the value of the function $\phi^+$ is successively
increased while in Step~2, the value of the function $\phi^-$ is successively decreased.
Precisely,
\begin{eqnarray*}
\phi^+ (x) \le \phi^+(x_P^+) \le \phi^+ (x_C^+) \, ,\\
\phi^- (x_C^+) \ge \phi^-(x_C^-) \ge \phi^- (x_P^-) \, .
\end{eqnarray*}

\subsubsection{Results and remarks}

In order to state the dynamic programming principle, we first
introduce admissible sets of points and half-spaces  for both
players. Precisely, we consider
\begin{equation}\label{defC}
  \mathcal{C}^{\pm}(x) = \{ (y,\varphi) \in B_\eps (x)  \times C^2 (\R^N) :
  \pm \varphi(y) \ge \pm \varphi(x)  \} ,
\end{equation}
\begin{equation}
  \label{defP+}
  \mathcal{P}^{+} (x,y,\varphi)  =
  \left\{ \begin{array}{ll}
      \{ z \in B_R(y) :  \varphi (z) \ge \varphi (y) \} & \text{ if }
      D\vp(y) \neq 0 \text{ and } \kappa^* [y,\varphi] > 0 \\
      \{x\} & \text{ if not} ,\end{array} \right.
\end{equation}
\begin{equation}
  \label{defP-}
  \mathcal{P}^{-} (x,y,\varphi)  =
  \left\{ \begin{array}{ll}
      \{ z \in B_R(y) :  \varphi (z) \le  \varphi (y) \} &  \text{ if }
      D\vp(y) \neq 0 \text{ and } \kappa_* [y,\varphi] <  0 \\
      \{ x \} & \text{ if not} .
\end{array} \right.
  \end{equation}
Hence, the dynamic programming principle associated to the game is
\begin{equation} \label{pdp:dislo}
  \ue(t,x)=  \sup_{ (x_P^+,\phi^+) \in \mathcal{C}^+ (x) } \left\{
    \inf_{ x_C^+ \in \mathcal{P}^+ (x,x_P^+,\phi^+) } \left\{
      \inf_{ (x_C^-, \phi^-) \in \mathcal{C}^- (x_C^+)} \left\{
        \sup_{x_P^- \in \mathcal{P}^- (x_C^+, x_C^-,\phi^-) } \left\{
          \ue\(t^-, x_P^- \) \right\} \right\} \right\} \right\}
\end{equation}
where
\begin{equation}\label{timereset:dislo}
  \left\{ \begin{array}{lr}
      t^+ = t &+ \left\{\begin{array}{lr}
          C_\eps(\ep \k^*[x_P^+,\G^+]^{-1})& \text{ if } D\phi^+ (x_P^+) \neq 0
          \text{ and } \k^*[x_P^+,\G^+] > 0,
          \\ \ep^2 & \text{ if not} ,
        \end{array} \right.
      \\
      t^- = t^+ &+  \left\{\begin{array}{lr}
          C_\ep(\ep |\k_*[x_C^-,\G^-]|^{-1}) & \text{ if } D\phi^- (x_C^-) \neq 0
          \text{ and } \k_*[x_C^-,\G^-] < 0  ,
          \\ \ep^2 & \text{ if not.} \end{array} \right.
\end{array} \right.
\end{equation}
The last main result is
\begin{theo} \label{theo:conv-dislo}
Assume that $u_T \in W^{2,\infty} (\R^N)$. Then the
sequence $\ue$ converges locally uniformly as $\ep \to 0$ towards the unique
viscosity solution of \eqref{eq:geom}, \eqref{eq:tc}.
\end{theo}
\begin{remark}
To avoid further technicalties, we assume that the terminal datum is
very regular; it can be shown that the result still holds true for much
less regular functions $u_T$ such as bounded uniformly continuous ones.
This extension is left to the reader.
\end{remark}

\subsection{Open problems}
As we mentioned above, the games we have constructed have much more
complicated rules than the Paul-Carol game for mean curvature
flow. It is then a natural open problem to find simpler games and in
particular a game for the integral curvature flow associated with the
singular measure $\nu(dz) = (1-\alpha) dz/|z|^{N+\alpha}$ which
converges (in some sense) as $\alpha \to 1$ to the original Paul and
Carol game. The reason to look for such as game is that it is known
\cite{imbert} that the integral curvature flow converges towards the
mean curvature flow as $\alpha \to 1$.  The same question can be
raised for the fractional Laplacian operators in the situation of
PIDE's: find a game a natural game associated to $\Delta^\alpha$
operators, which coincides with a natural game for $\alpha =1$.

\section{Proofs of convergence results}

\subsection{Proof of Theorem~\ref{theo:conv-pide}}
\label{subsec:conv-pide}

As explained in Remark~\ref{rem:stronger-result}, it is enough to
prove the convergence as $\eps \to 0$ by assuming that $\nu$ is
supported in $B_R$ for some fixed $R>0$.

For fixed $\eps >0$ and $x \in \R^N$, the value function $u^\eps(t,x)$
is finite for $t$ close to $T$ thanks to the following
lemma. Proposition~\ref{pro:icpide} below is needed to prove that
$u^\eps(t,x)$ is finite for all $t \in (0,T)$.
\begin{lem}[The functions $u^\eps$ are well defined]\label{lem:tech2}
For all $\Phi \in C^2(\mr^N)$ such that
\begin{equation}\label{restra}
\|\Phi\|_\infty \le \eps^{-\alpha},\  |D \Phi (x)|\le \ep^{-\alpha} ,\   |D^2 \Phi(x)|\le \eps^{-\alpha} \, ,
\end{equation}
we have
$$
-\eps F (t, x, D\Phi (x),D^2 \Phi (x), I_R [x,\Phi]) \le C \eps^\gamma
$$
with $\gamma = 1 -\alpha  \max (1,k_1,k_2)  \in (0,1)$ and $C$
depends on $F$, $\nu$ and $R$.
\end{lem}
\begin{proof}[Proof of Lemma~\ref{lem:tech2}]
We consider a bounded $C^2$ function $\Phi$ such that \eqref{restra} holds.
From the definition of $I_R[x,\cdot]$ (see \eqref{defI}), it is clear that there
exists a constant $C$ only depending on $R$ and $\nu$ such that
$$
|I_R[x,\Phi]| \le C \eps^{-\alpha} \, .
$$
We thus get from (A1) and (A3)
\begin{eqnarray*}
-\eps F (t, x, D\Phi (x),D^2 \Phi (x), I_R [x,\Phi]) &\le & C \eps (
1+ \eps^{-\alpha k_1} +  \eps^{-\alpha k_2} + \eps^{-\alpha})
\end{eqnarray*}
and the lemma follows at once.
\end{proof}
Let us define as usual the semi-relaxed limits  $\ub= \liminf^*_{\ep\to 0} \ue$ and $\uB=
\limsup^*_{\ep \to 0} \ue$.  Theorem \ref{theo:conv-pide} will follow  from the
 following two  propositions.
\begin{pro}\label{pro:pide}
The functions $\ub$ and $\uB$ are finite and are respectively a super-solution and a sub-solution
of \eqref{eq:pide}.
\end{pro}
\begin{pro}\label{pro:icpide}
  Given $R>0$, there exists a constant $C>0$ such that for all $\eps
  >0$, all $(t,x) \in (0,T) \times B_R $, we have
$$
|u^\eps (T-t,x) -u_T (x) | \le C t \, .
$$
In particular, $\ub$  and $\uB$ are finite and they satisfy at
time $t=T$ and for all $x \in \R^N$
$$
\ub (T,x) = \uB (T,x)= u_T(x) \, .
$$
\end{pro}
These two propositions together with the comparison principle imply
that $\ub = \uB$, \textit{i.e.} $u^\ep$ converges locally uniformly
towards a continuous function denoted $u$. This implies that $u$ is
a (continuous) viscosity solution of \eqref{eq:pide} satisfying
\eqref{eq:tc}. This finishes the proof of the theorem.

We now prove the two propositions.
\begin{proof}[Proof of Proposition~\ref{pro:pide}]
We use the general method proposed by Barles and Souganidis \cite{bs91} in order
to prove that $\ub$ is a super-solution of \eqref{eq:pide}.
This is the reason why, given a function $U :  \R^N \to \R$,
we introduce
\begin{multline*}
\Sc^\eps [U](t,x) = \sup_{\stackrel{\Phi \in
C^2(\mr^N)}{\|\Phi\|_{\infty},  |D \Phi (x)|, |D^2 \Phi (x)|\le \eps^{-\alpha} } }
\inf_{y \in B_R (x)} \bigg\{ U (y) + \Phi(x)- \Phi(y) \\
- \ep F( t,x,D\Phi(x), D^2 \Phi(x), I_R[x,\Phi]) \bigg\}  \, .
\end{multline*}
The two important properties of $\Sc^\eps$ are:
\begin{equation}
\label{mono1}\text{
it commutes with constants: $\Sc^\eps [U+ C] = \RE [U]+C$ for any constant $C \in \R$};\end{equation}
\begin{equation}\label{mono2}
\text{ it is monotone: if $U \le V$ then $\Sc^\eps [U] \le \Sc^\eps [V]$.}
\end{equation}
The dynamic programming principle \eqref{pdp:pide} is rewritten
as follows
\begin{equation}\label{pdp:pide-2}
\ue (t,x) = \Sc^\eps [\ue(t+\ep, \cdot) ](t,x)  .
\end{equation}

We now explain how to prove that $\ub$ is a super-solution of
\eqref{eq:pide}.  The case of $\uB$ is proven analogously thanks to a
``consistency lemma'' (see below Lemma~\ref{lem:consis-pide}).
Following Definition~\ref{def:visc-pide} and Remark~\ref{rem:supp-nu-compact},
consider  a $C^2$ bounded test function $\phi$ and a point
$(t_0,x_0)$ with $t_0 >0$ such that $\ub-\phi$ admits a strict minimum
$0$ at $(t_0,x_0)$ on $\mathcal{V}_0 = (0,T) \times B_{R+1}(x_0)$.  By
definition of $\ub$, there exists $(\tau_\ep, y_\ep) $
such that $(\tau_\ep, y_\ep) \to (t_0,x_0)$ and $u^{\ep} (\tau_\ep,
y_\ep)\to \ub(t_0,x_0)$ as $\ep \to 0$, up to a subsequence.  Let then
$(t_\ep, x_\ep)$ be a point of minimum of $\ue - \phi$ on
$\mathcal{V}_0$.  We have $$\ue(t_\ep, x_\ep)- \phi(t_\ep, x_\ep)\le
u^{\ep} (\tau_\ep, y_\ep) - \phi(\tau_\ep, y_\ep) \to \ub(t_0,x_0)-
\phi(t_0,x_0)=0$$ hence by definition of $\ub$ and $(t_0,x_0)$ as a
strict local minimum, we conclude that we must have $(t_\ep, x_\ep)
\to (t_0,x_0)$ as $\ep \to 0$.
 In addition, for all $(t,x) \in \mathcal{V}_0$, we have
$$
\ue (t,x) \ge \phi (t,x) + (\ue (t_\eps,x_\eps) - \phi (t_\eps,x_\eps)) =: \phi(t,x) + \xi_\ep .
$$
In particular, if $\eps$ is small enough, this inequality holds true
on $(0,T) \times B_R(x_\eps)$.  From the definition of $\Sc^\eps$,
the dynamic programming principle \eqref{pdp:pide-2}, and the properties
\eqref{mono1}--\eqref{mono2}, the previous inequality implies
$$
u^\eps (t_\eps,x_\eps) = \Sc^\eps [\ue (t_\ep+\ep, \cdot)
](t_\eps,x_\eps) \ge \Sc^\eps [\phi(t_\ep + \ep, \cdot ) + \xi_\ep](t_\eps,x_\eps) =
\Sc^\eps[\phi(t_\ep + \ep, \cdot)](t_\ep, x_\ep) + \xi_\ep  .
$$
Since $u^\ep(t_\ep, x_\ep)= \phi(t_\ep, x_\ep)+ \xi_\ep$, we get
$$
 \phi(t_\ep, x_\ep) \ge \Sc^\eps [\phi(t_\ep + \ep, \cdot)](t_\eps , x_\eps) .
$$
We claim Proposition~\ref{pro:pide} is proved if the following lemma
holds true.
\begin{lem}[Consistency lemma for PIDE]\label{lem:consis-pide}
  Consider a $C^2$ bounded test function $\psi$. Given a compact
  subset $K$ of $(0,T) \times \R^N$, there then exists a function
  $o(\eps)$ such that $o (\eps) / \eps \to 0$ as $\eps \to 0$ and for
  all $(t,x) \in K$, we have
$$
\Sc^\eps [ \psi ](t,x) = \psi (x) - \ep F (t,x,D\psi (x), D^2 \psi
(x), I_R [x,\psi])+ o(\eps) .
$$
\end{lem}
Indeed, applying this lemma to $\psi=\phi(t_\ep + \ep, \cdot)$, we are
led to
$$
\phi(t_\ep, x_\ep) \ge \phi( t_\ep + \ep, x_\ep) - \ep F(t_\ep,
x_\ep,D\phi (t_\eps + \ep,x_\eps), D^2 \phi (t_\eps +\ep,x_\eps), I_R
[x_\eps,\phi(t_\eps+\ep,\cdot)]) +o(\ep)\, .
$$
Dividing by $\ep$ and letting $\ep \to 0$ we obtain (using the $C^2 $
character of $\phi$ and continuity of $F$),
$$
- \p_t \phi(t_0, x_0) + F(t_0, x_0, D\phi(t_0, x_0), D^2
\phi(t_0,x_0), I_R[x_0,\phi(t_0, \cdot)]) \ge 0.
$$
This allows us to conclude that $\ub$ is a viscosity super-solution of
(\ref{eq:pide}). The proof that $\uB $ is a sub-solution is analogous.
\end{proof}
We now turn to the  proof of the consistency lemma, \textit{i.e.}
Lemma~\ref{lem:consis-pide}.
\begin{proof}[Proof of Lemma~\ref{lem:consis-pide}]
  This lemma easily follows from Lemma~\ref{lf1}.  Indeed,
  Lemma~\ref{lf1} implies that for $\psi \in C^2(\R^N)$, there exists
  $o (\eps)$ depending on $\psi$ such that
$$
\Sc^\eps [\psi](t,x) \le \psi (x) - \eps F (t,x,D\psi (x),D^2 \psi ( x), I_R[x,\psi]) + o (\eps).
$$
On the other hand, by choosing $\Phi = \psi $ in the definition of
$\Sc^\eps$, we immediately get
$$
\Sc^\eps [\psi](t,x) \ge \psi (x) - \eps F (t,x,D\psi (x),D^2 \psi ( x), I_R[x,\psi]) \, .
$$
Combining the two previous inequalities yield the desired result.
\end{proof}
We next turn to the proof of the crucial Lemma, \textit{i.e.}
Lemma~\ref{lf1}.
\begin{proof}[Proof of Lemma~\ref{lf1}]
  By contradiction assume this is wrong, hence there exists $\eta >0$
  and $\eps_n \to 0$ such that for all $y \in B_R(x)$,
\begin{multline}
\label{lf3} \Phi(x)- \psi(x) - \Phi(y)+ \psi(y)\\ >-
\ep_n\left(F(t,x,D\Phi(x),D^2 \Phi(x),I_R[x,\Phi]) - F(t,x,D\psi(x),
D^2 \psi(x), I_R[x,\psi])\right) +\eta \ep_n \, .
\end{multline}
In order to simplify notation, $\ep_n$ is simply denoted by $\ep$.

Let us first take $y= x+ \ep^{\hal} w$ with $\|w\|=1$. Inserting into
\eqref{lf3} and using the $C^2$ character of $\Phi$ and $\psi$ gives
$$
-\ep^{\frac12}  D(\Phi-\psi)(x) \cdot w \ge  - C \ep \,
$$ where $C$ depends only on $\Phi, \psi, F, x$ and not on $\ep$.
Dividing by $\ep^{\hal}$, and using the fact that this is true for
every $w \in \mr^N$ with $\|w\|=1$, we find that
\begin{equation} \label{dp}
|D (\Phi-\psi)(x)|\le C \ep^{\frac12} \, .
\end{equation}
Using now $y=x+\ep^{\frac13}w$ and doing a second order Taylor
expansion of $\Phi-\psi$, we find
$$
- \ep^{\frac13} D(\Phi-\psi)(x) \cdot w - \frac{ \ep^{\frac23}}{2}\langle
D^2(\Phi-\psi)(x)  w, w \rangle\ge  O(\ep)
$$
and using \eqref{dp} we obtain
$$
\langle  D^2(\Phi-\psi)(x)  w, w \rangle  \le O(\ep^{\frac16}) \, .
$$
Since this is true for any $w$ of norm 1, we deduce that
\begin{equation}\label{lmax}
  | D^2(\Phi-\psi)(x) | \le O(\ep^{\frac16})
\end{equation}
where $|A| = \lambda_{max}(A)$ denotes here the largest eigenvalue of a
symmetric matrix.
Finally setting $y=x+z$ in \eqref{lf3} with $z \in B_R(0)$ gives
$$
\Phi(x+z)-\Phi(x)\le \psi(x+z)- \psi(x) + C \ep
$$
where the constant $C$ again depends on $x,\Phi,\psi,F$ but not on $z$.
Integrating this inequality with respect to $z$'s such that
$\gamma \le |z| \le R$ (with $\gamma >0$ to be chosen later), we
obtain
$$
\int_{|z| \ge \gamma} ( \Phi(x+z)-\Phi(x) ) \nu (dz ) \le \int_{ |z|
  \ge \gamma} ( \psi(x+z)-\psi(x) ) \nu (dz ) + C \eps \nu (|z| \ge
\gamma) \, .
$$
By \eqref{cond:nuPIDE}, we know that $\gamma^2 \nu ( |z| \ge \gamma)
\le C_\nu$ where $C_\nu$ is a constant that only depends on $\nu$,
we conclude that
\begin{equation}\label{estim1}
  \int_{ |z| \ge \gamma} ( \Phi(x+z)-\Phi(x) ) \nu (dz ) \le \int_{
    |z| \ge \gamma} ( \psi(x+z)-\psi(x) ) \nu (dz ) + C \eps
  \gamma^{-2}
\end{equation}
where  $C$ depends on $x,\Phi,\psi,F$ and $\nu$ (we do not
change the name of the constant). On the other hand, by using the $C^2$
regularity of $\Phi$ and $\psi$, we obtain
\begin{multline*}
  \int_{|z| \le \gamma} ( \Phi(x+z)-\Phi(x) - D \Phi (x) \cdot z)
  \nu (dz ) \le \int_{ |z| \le \gamma} ( \psi(x+z)-\psi(x) - D\psi (x)
  \cdot z)
  \nu (dz ) \\
  + \( \frac12 | D^2 \Phi (x) - D^2 \psi (x) | + C \gamma \)\int_{ |z|
    \le \gamma} |z|^2 \nu (dz)
\end{multline*}
where $C$ only depends on $\Phi$ and $\psi$. Now choosing $\gamma
= \eps^{\frac16}$ and using \eqref{lmax}, we   get
\begin{multline*}
  \int_{|z| \le \gamma} ( \Phi(x+z)-\Phi(x) - D \Phi (x) \cdot z)
  \nu (dz ) \le \int_{
    |z| \le \gamma} ( \psi(x+z)-\psi(x) - D \psi (x) \cdot z) \nu (dz ) \\
  + C\eps^{\frac16}\,
\end{multline*} where $C$ depends only on $  \Phi, \psi, x, F, \nu$.
Finally, from \eqref{lmax}, we have
$$
\left|\int_{ \ep^{1/6} \le |z|\le 1} ( - D\Phi(x) \cdot z + D \psi(x)
  \cdot z ) \nu(dz) \right|\le C \ep^{1/6}.
$$

Combining the above estimates with \eqref{estim1}, we conclude that
\begin{equation}\label{i}
  I_R[x,\Phi]\le I_R[x,\psi]+C\ep^{\frac16}
\end{equation}
where $C$ depends on $\Phi,\psi, F,x$ and $\nu$. Combining \eqref{dp},
\eqref{lmax} and \eqref{i}, the continuity of $F$ and its monotonicity
condition \eqref{cond:ellip} yield that
\begin{equation}\label{compf}
  F(t,x,D\Phi(x),D^2\Phi(x),I_R[x,\Phi]) -F(t,x,D\psi(x),D^2\psi(x),I_R[x,\psi]) \ge o(1) \, .
\end{equation}
Inserting this into \eqref{lf3} and choosing $y=x$, we find
$$0>\ep o(1)+ \eta \ep$$
a contradiction for $\ep$ small enough. Hence the lemma is proved.
\end{proof}
We conclude the proof of the convergence theorem
(Theorem~\ref{theo:conv-pide}) by proving that the terminal condition
is satisfied (Proposition~\ref{pro:icpide}).
\begin{proof}[Proof of Proposition~\ref{pro:icpide}]
It is enough to prove that for some constant $C>0$ and for all $k
\in \mathbb{N}$, we have
\begin{equation}\label{estim:step}
  \forall (t,x) \in (0,T) \times \R^N,
  \quad |u^\eps (T- k\eps, x) - u_T ( x) | \le C k \eps \, .
\end{equation}
We argue by induction. The relation ~\eqref{estim:step} is clear for $k=0$.
We assume it is true for $k$ and we prove it for $k+1$.

We first consider a family $u_T^\eta$ of bounded $C^2$ functions
such that $\| u_T^\eta - u_T \|_{W^{2,\infty}} \le \eta$. From the
one-step dynamic programming principle \eqref{pdp:pide} and the
choice $\Phi = u_T^\eta$, we easily deduce  that
\begin{multline*}
  u^\eps (T-(k+1)\eps ,x) \\ \ge \inf_{y \in B_R (x)} \{ u^\eps (T- k
  \eps ,y) + u_T^\eta (x)- u_T^\eta (y)
  - \eps F(T-(k+1)\eps,x, D u_T^\eta (x), D^2 u_T^\eta (x) , I_R[x,u_T^\eta]) \} \\
  \ge \inf_{y \in B_R (x)} \{ u^\eps (T-k\eps,y) - u_T (y) \} + u_T
  (x) - C_1 \eps - 2 \eta \\
  \ge  - C k\eps  + u_T (x)- C_1 \eps - 2 \eta \end{multline*}
where we used \eqref{estim:step}  and we chose
$$
C_1 =  \max \{ F (t,x,p,A,l) : t\in (0,T), x \in B_R, |p| + |A|
\le 2 \| u_T\|_{W^{2,\infty}}, |l| \le C_\nu \|
u_T\|_{W^{2,\infty}} \} +1 \, .
$$
Changing $C$ if necessary in \eqref{estim:step} we can assume that $C\ge C_1$.
Since $\eta$ is arbitrary, we easily get an estimate from below:
$$u^\ep (T (k+1)\eps ,x)  - u_T(x) \ge - C(k+1)\ep
.$$
Using once again  the one-step dynamic programming principle
\eqref{pdp:pide} and \eqref{estim:step}, we next get
\begin{multline}
\label{estim:ueps-1}
u^\eps (T- (k+1)\eps,x) \\
\le \sup_{\stackrel{\Phi \in C^2(\mr^N)}{\|\Phi\|_\infty, |D \Phi (x)|, |D^2 \Phi
    (x)|\le \eps^{-\alpha} } } \inf_{y \in B_R (x)}
\bigg( u_T^\eta (y) + \Phi (x) - \Phi
(y) \\ - \eps F (T-(k+1)\eps, x, D\Phi (x),D^2 \Phi (x), I_R [x,\Phi])
\bigg)  + C k \eps + \eta \,.
\end{multline}
In order to get the upper bound in \eqref{estim:step}, we use the
consistency Lemma~\ref{lem:consis-pide}.  Applying it to $\Phi$, $\psi
= u^\eta_T$, $t= T- (k+1)\eps$, we get from \eqref{estim:ueps-1}
\begin{eqnarray*}
u^\eps (T- (k+1)\eps,x) &\le & u_T(x) - \eps ( F (T-(k+1)\eps,
x, D u_T^\eta, D^2 u_T^\eta (x), I_R [x, u_T^\eta]) + o(1))  \\ && + C k \eps + 2\eta  \\
& \le & u_T (x) + C_1\ep+ C k\eps + 2\eta\le u_T(x)+ C(k+1) \ep + 2 \eta
\end{eqnarray*}
and since $\eta$ is arbitrary, the proof of the proposition is now
complete.
\end{proof}

\subsection{Proof of Theorem~\ref{theo:conv-eikonal}}\label{sec3.2}

  We first remark that for all $\eps >0$ and all $(t,x) \in (0,T]
  \times \R^N$, $\inf u_T \le \ue (t,x) \le \sup u_T$. We thus can
  consider the upper and lower relaxed limits $\overline{u}$ and
  $\underline{u}$ (they are finite) and we will prove below the
  following result.
  \begin{pro}\label{lem:eik-sup}
    The functions $\underline{u}$ and $\overline{u}$ are respectively
    a super-solution and a sub-solution of \eqref{eq:eikonal}.
  \end{pro}
  In order to conclude that $\ue$ converges towards the unique
  solution of \eqref{eq:eikonal}, \eqref{eq:tc}, it is then enough to
  prove that $\overline{u} (T,x) \le u_T (x)    \le \underline{u} (T,x).$
  This is an easy consequence of the following proposition whose proof
  is postponed too.
  \begin{lem}\label{lem:tc}
    Given $\delta,R>0$ there exists $C>0$ such that for all $t \in
    (T-\delta, T]$ and all $x \in B_R(0)$
    \begin{equation}\label{estim:tc}
      |\ue (t,x) - u_T (x) | \le C (T-t + \eps^{\frac12})\; \|u_T\|_{Lip}.
    \end{equation}
  \end{lem}
  The comparison principle for \eqref{eq:eikonal} then permits to
  conclude.
\medskip

It remains to prove Proposition~\ref{lem:eik-sup} and
Lemma~\ref{lem:tc}. In order to do so, we proceed as in the proof of
Proposition~\ref{pro:pide} by introducing an operator $\Sc^\eps
[\phi]$. More precisely, if $\phi :(0,T] \times \R^N \to \R$
is a bounded function, we let
\begin{equation}\label{defS}
\Sc^\eps [\phi](t,x) = \sup_{ x_P\in  E^+(x) }
\left\{ \inf_{ x_C \in E^-(x_P)  }
 \phi (t_C,x_C) \right\}
\end{equation}
where $t_C$ is defined by \eqref{eq:timereset} and $E^\pm$
is defined by (\ref{EP}).  It is convenient to write
\begin{eqnarray*}
t_ P &=& t + T_P (x,x_P) \\
t_C &=& t_P + T_C (x_P,x_C)
\end{eqnarray*}
where $T_P$ and $T_C$ are defined as follows
\begin{eqnarray}\label{tp}
  T_P (x,x_P)=
  \left\{\begin{array}{ll} C_\ep [ \eps (\v_+(x_P))^{-1} ] &
      \mbox{ if } B_\ep(x) \cap\{ \v>0\} \neq \emptyset \\
      \eps^{2} & \mbox{ if not } ,\end{array}\right.\\
  \nonumber T_C (x_P,x_C)=
  \left\{ \begin{array}{ll}
      C_\ep[ \eps (\v_-(x_C))^{-1} ]  & \mbox{ if } B_\ep(x_P) \cap \{ v<0\}\neq \emptyset\\
      \eps^{2} & \mbox{ if not} \, .
        \end{array}\right.
\end{eqnarray}
We also introduce
\begin{eqnarray*} \RE [\phi](t,x)&=& \sup_{x_P\in E^+(x)}\phi(t_P, x_P)\\
\Re[\phi](t,x)&=& \inf_{x_C\in E^-(x_P)} \phi(t_C,x_C).
\end{eqnarray*}
The two important properties of $\RE$ are:
\begin{eqnarray}
  \label{mono3}
 \text{it commutes with constants:}&& \RE [\phi+ c] = \RE [\phi]+c \text{ for any constant }
  c \in \R;\\
\label{mono4}
 \text{ it is monotone:} && \text{if } \phi \le \psi \text{ then } \RE [\phi] \le \RE [\psi].
\end{eqnarray}
We now rewrite $\Sc_\eps$ as
$$
\Sc_\ep[\phi]= \Re[\RE[\phi]] \, .
$$
We remark that $\Re$ and $\Sc_\eps$ also commute with constants and
are monotone.
One can also observe that $\Re$ and $\RE$ have opposite values if
$v$ is changed into $-v$ and $\phi $ into $-\phi$.

Recall that the  dynamic programming principle  in this context is
\eqref{pdp:eikonal} \textit{i.e.} with the new terminology
\begin{equation}
\label{pdp:eikonal-bis}
u^\ep(t,x)= \Sc_\ep[u^\ep](t,x)= \Re[\RE [u^\ep]](t,x).\end{equation}

The core of the proof of  of Proposition~\ref{lem:eik-sup} lies in the
following ``consistency lemma''.
\begin{lem}[Consistency lemma for the eikonal equation]\label{lem:consis-eik}
  Consider a $C^1$ bounded smooth function $\phi: (0,T] \times \R$.
  Given $r>0$ small enough and $(t_0,x_0) \in (0,T-r) \times \R^N$,
  there exists a function $o(1)$ depending only on $(\eps,r)$, $\phi$, $(t_0,x_0)$
  and the speed function $v$ such that $o (1) \to 0$ as $(\eps, r) \to
  0$, and the following holds:  for all $(t,x) \in B_r(t_0,x_0)$, there exists
  $x_P,y_P, x_C, y_C \in B_\eps (x)$, such that
\begin{eqnarray}\label{eik-mino}
\RE [\phi](t,x)- \phi (t,x)
  &\le& T_P (x,x_P)  \bigg(\partial_t \phi (t,x)
  +  v_+(x) | D \phi (t,x)|  +o(1) \bigg)  \, , \\
\label{eik-majo}
  \RE [\phi](t,x)-\phi(t,x) &\ge& T_P(x, y) \bigg(\partial_t \phi (t,x)
  +  v_+(x) | D \phi (t,x)|  +o(1) \bigg)  \, ,
\end{eqnarray}
and
\begin{eqnarray}\label{eik-mino1}
  \Re [\phi](t,x)- \phi (t,x)
  &\le& T_C (x,y_C)  \bigg(\partial_t \phi (t,x)
  +  v_-(x) | D \phi (t,x)|  +o(1) \bigg)  \, , \\
\label{eik-majo1}
\Re [\phi](t,x)- \phi (t,x)
&\ge& T_C (x,x_C)  \bigg(\partial_t \phi (t,x) +  v_-(x) | D \phi (t,x)|
+o(1) \bigg)  \, .
\end{eqnarray}
\end{lem}
We can deduce from this lemma the following one
\begin{lem}[Consistency lemma for the eikonal equation - second
  version]\label{lem:consis-eik2} We consider a function
   $\phi: (0,T] \times \R$ that is bounded and $C^1$ and
  $(t_0,x_0) \in (0,T) \times \R^N$.  There exist a function $o(1)$
  such that $o (1) \to 0$ as $(\ep, r) \to 0$ and positive numbers
  $m_\ep$, $M_\ep$ such that for all $(t,x) \in B_r(t_0,x_0)$ and all
  $(\eps, r)$ small enough
\begin{multline}\label{eq:consis-eik}
m_\ep \bigg(\partial_t \phi (t_0,x_0) +  v(x_0) | D \phi (t_0,x_0)|  +o(1)\bigg)  \le
\Sc^\eps [\phi](t,x)- \phi (t,x)\\
\le M_\ep \bigg(\partial_t \phi (t_0,x_0) +  v(x_0) | D \phi (t_0,x_0)|  +o (1) \bigg)  \, .
\end{multline}
\end{lem}
The proofs of the two previous lemmas are postponed.
We first explain how to derive Proposition~\ref{lem:eik-sup}
from Lemma~\ref{lem:consis-eik2}.
\begin{proof}[Proof of Proposition~\ref{lem:eik-sup}]
  We are going  show that $\ub $ is a
  super-solution. Following Definition~\ref{defi2}, let $\phi$ be a
  $C^1$ function such that $\ub-\phi$ admits a minimum $0$ at
  $(t_0,x_0)$ on $\mathcal{V}_0 = B_\delta(t_0,x_0)$. Without loss
  of generality, we can assume that this minimum is strict, see
  \cite{cil92}.  Arguing as in the proof of Proposition
  \ref{pro:pide}, we deduce that $u^\ep-\phi$ admits a minimum at
  $(t_\ep, x_\ep) $ on $\mathcal{V}_0$ with $(t_\ep, x_\ep) \to (t_0,
  x_0)$ as $\ep \to 0$; and for all $(t,x) \in \mathcal{V}_0$
$$
u^\ep( t,x) \ge \phi(t,x) + (u^\ep(t_\ep, x_\ep)- \phi(t_\ep, x_\ep))
: = \phi(t,x) + \xi_\ep
$$
From the properties \eqref{mono3}, \eqref{mono4} of $\Sc^\ep$ and the dynamic programming principle
(\ref{pdp:eikonal-bis}), we have
$$
u^\ep(t_\ep,x_\ep)\ge \Sc^\ep[ \phi + \xi_\ep](t_\ep,x_\ep) =
\Sc^\ep[ \phi](t_\ep,x_\ep) + \xi_\ep.
$$
Since $u^\ep(t_\ep, x_\ep) = \phi(t_\ep, x_\ep) + \xi_\ep$ it follows that
$$
\phi(t_\ep, x_\ep) \ge \Sc^\ep[ \phi](t_\ep,x_\ep).
$$
Using Lemma \ref{lem:consis-eik2} applied at $(t_\ep, x_\ep)$ we deduce the existence of $m_\ep$ such that
$$
 m_\ep \( \p_t\phi(t_0, x_0) + v(x_0) |D\phi(t_0, x_0)|+ o (1)\)\le 0.
$$
Dividing by $m_\ep>0$ then letting $\ep \to 0$ and $r\to 0$, we obtain
$$
 \p_t\phi(t_0, x_0) + v(x_0) |D\phi(t_0, x_0)| \le 0.
$$
We thus conclude that $\ub$ is a super-solution. The proof that $\uB$
is a sub-solution is entirely parallel and
Proposition~\ref{lem:eik-sup} is proved.
\end{proof}

We now turn to the core of the argument, \textit{i.e.}
\begin{proof}[Proof of Lemma~\ref{lem:consis-eik}] First we note that
  it suffices to prove \eqref{eik-mino}, \eqref{eik-majo}, because
  \eqref{eik-mino1} and \eqref{eik-majo1} follow by changing $\phi$
  into $- \phi$ and $v$ into $-v$.

Consider $(t,x)\in B_r(t_0,x_0)$ and  $x_P$ an $\eps^3$-optimal position starting from $x$
at time $t$ i.e. such that  $x_P \in E^+(x)$ and
$$
\RE [\phi] (t,x) = \phi (t_P, y) + O (\eps^3) \, .
$$
First, remark that $x_P \in B_\eps (x) \subset B_{r+\eps} (x_0)$.  In
particular, $|v(x_P)| \le |v(x_0)| + L_v (r+\eps)$ where $L_v$ is the
Lipschitz constant of the function $v$. Hence, $|v(y)| \le
\frac1{\eps^{\frac12}}$ for $\eps$ small enough (only depending on
$L_v$, $r$ and $x_0$).  In particular in view of the definitions
\eqref{cep} and \eqref{tp}, $T_P(x,x_P) = \eps^2$, $T_P(x,x_P) = \eps
v_+(x_P)^{-1}$ or $T_P(x,x_P) = \eps^{\frac12}$.

We now distinguish these three cases.
\paragraph{Case $T_P(x,x_P )= \eps^2$.} This can only happen  if $x_P=
x$, $v_+(x) = 0$ and $t_P = t + \eps^2$, so  we may write, by Taylor expansion of $\phi$,
\begin{eqnarray*}
\RE [\phi] (t,x) - \phi (t,x) &=&
 \phi(t_P,x_P)- \phi(t,x)+O(\ep^3)\\
 & = &
 \phi (t+\eps^2,x) - \phi (t,x)+O(\ep^3)  \\
&=& \eps^2 (\partial_t \phi (t,x) + v_+ (x) |D \phi (t,x)| + o_\eps (1))
\end{eqnarray*}
and we obtain the desired result \eqref{eik-mino}--\eqref{eik-majo}.

\paragraph{Case $T_P (x,x_P) = \eps^{\frac12}.$} This case happens only if    $0 <
v(x_P) \le \eps^{\frac12}$.  This implies in particular that
$|v(x)| \le \eps^{\frac12} + L_v \eps$.  Then we simply write
\begin{eqnarray*}
\RE [\phi] (t,x) - \phi (t,x) &=& \phi (t+\eps^{\frac12},x_P) - \phi (t,x) + O (\eps) \\
& = & \eps^{\frac12} (\partial_t \phi (t,x) + o_\eps (1)) \\
& = & \eps^{\frac12} (\partial_t \phi (t,x) + v_+(x) |D \phi (t,x)| + o_\eps (1))
\end{eqnarray*}
and we obtain the desired result in this case too.

\paragraph{Case $T_P (x,x_P) = \eps v (x_P)^{-1}$.} Then
$v(x_P) \ge \eps^{\frac12}$.  This implies in particular that $v(x)
\ge \eps^{\frac12} - L_v \eps$.  We write in this case, Taylor expanding $\phi$ again
\begin{multline*}
\RE [\phi] (t,x) - \phi (t,x) =
 \phi(t_P, x_P)- \phi(t,x)+O(\ep^3)
  \\
 =  \frac{\eps}{v (x_P)} (\partial_t \phi (t,x) + o_\eps (1))
+ (x_P -x) \cdot D \phi (t,x)+ O(\eps^2) \, .
\end{multline*}
Hence, we are done if $D \phi (t,x)=0$. If not, we can write
\begin{eqnarray*}
  \RE [\phi] (t,x)- \phi (t,x)
  & \le & \frac{\eps}{v (x_P)} (\partial_t \phi (t,x) + o_\eps (1))
  + \eps |D \phi (t,x)| + O (\eps^2) \\
  & \le &   \frac{\eps}{v (x_P)} (\partial_t \phi (t,x) + v_+ (x) |D \phi (t,x)|+ o_\eps (1))
  \, .
\end{eqnarray*}
To get the reversed inequality in this last case, we consider $y_P = x
+ \eps \frac{D\phi (t,x)}{|D \phi (t,x)|}$. Remark that $v(y_P) \ge
v(x) - L_v \eps \ge \eps^{\frac12} - 2 L_v \eps>0$ for $\ep $ small
enough (only depending on $L_v$). Then, since $y_P \in E^+(x)$, we
have
\begin{eqnarray*}
\RE [\phi] (t,x)- \phi (t,x)
& \ge &  \phi (t+ T_P (x,y_P),y_P) - \phi (t,x) \\
& \ge & T_P (x,y_P) (\partial_t \phi (t,x) + o_\eps (1)) + \eps |D \phi (t,x)| + O (\eps^2) \\
& \ge & T_P (x,y_P) (\partial_t \phi (t,x) + \frac{\eps}{ T_P (x,y_P)} |D \phi (t,x)| + o_\eps (1)) \\
& \ge & T_P (x,y_P) (\partial_t \phi (t,x) + \max (\eps^{\frac12}, v(y_P)) |D \phi (t,x)| + o_\eps (1)).
\end{eqnarray*}
If $v(y_P) \ge \eps^{\frac12}$, then we use the Lipschitz continuity of $v_+$ in order to get
$$
\RE [\phi] (t,x)- \phi (t,x) \ge T_P (x,y_P) (\partial_t \phi (t,x) +
v_+ (x) | D \phi (t,x)| + o_\eps (1) ).
$$
If $v({y}_P) < \eps^{\frac12}$, then $\ep^{\frac12} - L_v \eps \le v (x) \le \eps^{\frac12} + L_v \eps $
and we conclude that $v (x) =o(1)$ and the result is obtained in this case too.
\end{proof}
\begin{proof}[Proof of Lemma~\ref{lem:consis-eik2}]
Recall that $\Sc_\ep[\phi]= \Re[\RE[\phi]]$. We distinguish cases.

\paragraph{Case $v(x_0)>0$.} In this case, we can write $v(x_0) \ge
2\delta_0 >0$.  For $\eps$ small enough and $r\le \hal$, we have for
all $x_P \in E^+ (x) \subset B_\eps (x)$, $\delta_0 \le v(x_P) \le
v(x_0) + 1$ and $T_P(x,x_P) = \frac{\ep}{v(x_P)}$. We thus obtain from
Lemma~\ref{lem:consis-eik} for all $(t,x) \in B_r (t_0,x_0)$, the
existence of $x_P\in B_\ep(x)$ such that
\begin{eqnarray*}
  \RE [\phi] (t,x) &\le& \phi (t,x) + \frac{\eps}{v(x_P)} (\partial_t \phi (t,x) + v_+(x) |D\phi (t,x)|
  + o(1))  \\
  & \le &  \phi (t,x) + \frac{\eps}{\delta_1} (\partial_t \phi (t_0,x_0) + v(x_0) |D\phi (t_0,x_0)|
  + o(1))
\end{eqnarray*} as $(\ep, r) \to 0$,
where
$$
\delta_1 =
\left\{ \begin{array}{ll}
\delta_0 & \text{ if } \partial_t \phi (t_0,x_0) + v(x_0) |D\phi (t_0,x_0)|>0, \\
v(x_0)+1 & \text{ if } \partial_t \phi (t_0,x_0) + v(x_0) |D\phi (t_0,x_0)|\le 0 .
\end{array}
\right.
$$
Since $\Re$ commutes with constants and is monotone, the previous
inequality implies the following one
\begin{eqnarray*}
  \Sc_\eps [\phi] (t,x) -\phi (t,x) &\le& \Re [\phi] (t,x) - \phi (t,x) \\
  && + \frac{\eps}{\delta_1} (\partial_t \phi (t_0,x_0) + v(x_0) |D\phi (t_0,x_0)|
  + o(1)). \end{eqnarray*}

  Now from  \eqref{eik-mino1} we deduce
$$
\Re[\phi](t,x) \le T_C(x,y_C) \( \p_t \phi(t_0,x_0) + o(1)\)
$$
for some $y_C \in B_\ep(x)$ and since $v(x_0)\ge 2\delta_0>0$ we have
$B_\ep(x) \cap \{v<0\}= \varnothing$ for $r$ small enough and thus
$T_C(x,y_C)=\ep^2$. It follows that
  \begin{eqnarray*}
    \Sc_\eps [\phi] (t,x) -\phi (t,x)
    &\le  &  \eps^2 \partial_t \phi (t_0,x_0) + o (\eps^2) \\
    & & + \frac{\eps}{\delta_1} (\partial_t \phi (t_0,x_0)
    + v(x_0) |D\phi (t_0,x_0)|  + o(1)) \\
    & \le &  \frac{\eps}{\delta_1} (\partial_t \phi (t_0,x_0)
    + v(x_0) |D\phi (t_0,x_0)|  + o(1)) \, .
\end{eqnarray*} which establishes  the upper bound part in \eqref{eq:consis-eik}.
The case $v(x_0)<0$ is analogous.

\paragraph{Case $v(x_0)=0$.}
By \eqref{eik-mino}, we may write
 \begin{eqnarray*}
   \RE[\phi](t,x)& \le &  \phi(t,x) + T_P(x,x_P) (\p_t\phi(t,x)+o(1))\\
   & \le & \phi(t,x)+ M_{\ep,1}(\p_t\phi(t,x)+o(1))
\end{eqnarray*}
for some positive constant $M_{\ep,1}$, since $T_P$ is bounded above
and below by positive constants depending on $\ep$ (the last relation
is obtained by discussing according to the sign of $\p_t\phi(t,x)$).
From \eqref{eik-mino1} we obtain similarly that
$$\Re[\phi](t,x)\le \phi(t,x) + M_{\ep,2} (\p_t\phi(t,x)+o(1)) $$ for some positive $M_{\ep,2}$.
 Combining the two relations we obtain
 $$\Sc^\ep[\phi](t,x)- \phi(t,x)\le (M_{\ep,1}+M_{\ep,2})(\p_t\phi(t,x)+o(1)).$$ The lower bound is entirely parallel,
  and the desired result follows in this case.
The proof is now complete.
\end{proof}

\begin{proof}[Proof of Lemma \ref{lem:tc}]
  It is enough to study the sequence of iterated positions and times
  starting from $(t,x) \in (T-\delta,T] \times B_R (0)$. The dynamic
  programming principle (\ref{pdp:eikonal}) gives optimal positions
  $x_P$, $x_C$ such that $\ue(t,x)= \ue(t_C, x_C)$ for the
  corresponding time $t_C$.  Letting $t_0=t$ and $x_0=x$, and
  iterating this, we may define for $k \in \{1,\dots, K\}$ optimal
  positions $x_k$ (corresponding to the $x_C$) and corresponding times
  $t_k$ with $K$ such that $t_K \ge T$ and $t_{K-1} < T$ such that
$$
\ue(t_0, x_0)= \ue(t_k, x_k)= \ue(t_K, x_K) = u_T (x_K).
$$
It follows that
$$
\ue (t,x) -u_T(x) = u_T(x_K) - u_T (x) \; .
$$
We conclude from the previous equality that, in order to prove
\eqref{estim:tc}, it is enough to prove that for any $k \in \{0,K-1\}$
\begin{equation}\label{estim:discrete-speed}
| x_{k+1} - x_{k} | \le C (t_{k+1} - t_{k})
\end{equation}
with $C$ not depending on $\eps$ and $(t,x)$ (but possibly on $\delta$ and
$R$).

We notice that the supremum and the infimum defining $\ue$ may not be
attained. In this case, we simply choose first an $\eps^2$-optimal
position, then an $\eps^3$-optimal position and iterating this, we obtain
an error which is smaller than $\eps$ as soon as $\eps \le \frac12$.

In order to prove such a result, we first remark that it suffices to
prove that
$$
|x_P - x | \le C (t_P-t) \quad \text{ and } \quad |x_C -x_P|\le C
(t_C-t_P) \; .
$$
Hence, when $x_P=x$ and $x_C=x_P$, this is automatically satisfied.
If not, we always have $|x_P -x|\le \ep$ and $|x_C -x_P|\le \ep$.
Hence, we only need to check that for such time steps
$$
C_\ep ( \ep |\v(x_P)|^{-1} ) \ge C^{-1} \eps \quad \text{ and }
\quad C_\ep ( \ep |\v(x_C)|^{-1} ) \ge C^{-1} \ep
$$
for $C>0$ well chosen. This is equivalent to showing
$$
 |\v(x_i)| \le C  \quad \text{ for } i=P,C \; .
$$
By Lipschitz continuity of $v$,  there exists $C_R$ such that for all $y
\in B_{R+1} (0)$
 $$
 |\v(y)| \le C_R \; .
 $$
Recall that $x \in B_R (0)$.
If the finite sequence $(x_k)_{k=0, \dots, K}$ remains in $B_{R/2}
(x) \subset B_{R+1} (0)$, we are done: we choose $C=C_R$.
We now claim that the finite sequence does remain in $B_{R/2}$.

We argue by contradiction.
If not, consider $k_0$,  the smallest integer $k\le K$ such that $x_k \in
B_{R} (x) \setminus B_{R/2}(x) \subset B_{R+1} (0)$. Consider also
the number  $k_1$ ($\le k_0$) of steps such that Paul and Carol move
(\textit{i.e.} $x_P \neq x$ and $x_C \neq x_P$). This implies that the
corresponding time increments are at least $\eps / C_R$. This also
implies that $2 k_1 \eps \ge R/2$. Indeed,
$$
\frac{R}2 \le |x_{k_0} - x | \le \sum_{i=0}^{k_0-1} |x_{i+1} -x_i|
\le k_1 \times (2 \ep) \; .
$$
Recalling that $t \in (T-\delta, T]$, it is now enough to choose
$$
\delta < \frac{R}{4C_R}
$$
to conclude that $ t_{k_0}- t\ge \frac{k_1\ep}{C_R} \ge \frac{R}{4C_R} >\delta $ and thus  $t_{k_0} > T$, and get a contradiction.
\end{proof}


\subsection{Proof of Theorem~\ref{theo:conv-dislo}}

First we denote $\ub= \liminf_* \ue$ and $\uB= \limsup^* \ue$. These relaxed semi-limits
are finite since we always have $\inf u_T \le \ue \le \sup u_T$ and
$u_T$ is assumed to be uniformly bounded.
 The theorem follows as above from the following two results
\begin{pro}\label{lem:dislo-semi}
The functions $\ub$ and $\uB$ are respectively a super-solution and a sub-solution
of \eqref{eq:dislo}.
\end{pro}
\begin{lem}\label{lem:dislo-ci}
Given $R, \delta$, there exists $C>0$ such that for all $t \in
(0,\delta)$ and all $x \in B_R(0)$
\begin{equation}\label{estim:tc-dislo}
|\ue (t,x) - u_T (x) | \le C (T-t + \eps^{\frac12})\; .
\end{equation}
\end{lem}
Lemma~\ref{lem:dislo-ci} implies that $\ub (T,x) \ge u_T (x) \ge \uB (T,x)$ and the comparison
principle for \eqref{eq:dislo} (see \cite{imbert}) permits to conclude.

It remains to prove Proposition~\ref{lem:dislo-semi} and Lemma~\ref{lem:dislo-ci}.
We first introduce some notation, analogous to that of Section \ref{sec3.2}. Given $x \in \mr^N$ and $\phi \in C^2$,
$T_P$ and $T_C$ are defined by
\begin{eqnarray}\label{tctp}
  T_P (x,\phi) = \left\{\begin{array}{ll} C_\eps(\ep \k^*[x,\phi]^{-1})& \text{ if }
D \phi (x) \neq 0 \text{ and }  \k^*[x,\phi] > 0
      \\ \ep^2 & \text{ if not} \end{array} \right.
  \\ \nonumber
  T_C (x,\phi) =  \left\{\begin{array}{ll}
      C_\ep(\ep |\k_*[x,\phi]|^{-1}) & \text{ if } D \phi (x) \neq 0 \text{ and } \k_*[x,\phi] < 0 \\
      \ep^2 & \text{ if not} \, . \end{array} \right.
\end{eqnarray}
It is convenient
to write
\begin{eqnarray*}
  t^+ &=& t + T_P (x_P^+,\phi^+) \,  \\
  t^- & =& t^+ + T_C (x_C^-,\phi^-)
\end{eqnarray*}
where $T_P$ and $T_C$ are defined by (\ref{tctp}).

We now introduce for any arbitrary function $\phi: (0,T) \times \R^N
\to \R$ the following operator
$$
\Sc^\eps [\phi](t,x) = \sup_{ (x_P^+,\phi^+) \in \mathcal{C}^+ (x) }
\left\{
    \inf_{ x_C^+ \in {\mathcal{P}}^+ (x,x_P^+,\phi^+) } \left\{
      \inf_{ (x_C^-, \phi^-) \in \mathcal{C}^- (x_C^+)} \left\{
        \sup_{x_P^- \in {\mathcal{P}}^- (x_C^+, x_C^-,\phi^-) } \left\{
          \phi (t^-, x_P^- ) \right\}\right\}\right\}\right\}
$$
where $t^-$ is defined in \eqref{timereset:dislo}.
The dynamic programming principle~\eqref{pdp:dislo} can be rewritten as follows
\begin{equation} \label{pdp:dislo-bis}
u^\eps(t,x) = \Sc^\eps [u^\eps](t,x).
\end{equation}

 For the reader's convenience, we recall here
the definitions of $\mathcal{C}^\pm(x)$ and $\mathcal{P}^\pm (x,y,\varphi)$:
\begin{eqnarray*}
  \mathcal{C}^{\pm}(x) &=& \{ (y,\varphi) \in B_\eps (x)  \times C^2 (\R^N) :
\pm \varphi(y) \ge \pm \varphi(x)  \} \, , \\
  \mathcal{P}^{+} (x,y,\varphi) & =&
\left\{ \begin{array}{ll}
\{ z \in B_R(y) :  \varphi (z) \ge \varphi (y) \} & \text{ if } D\varphi(y) \neq 0 \text{ and }   \kappa^* [y,\varphi] > 0 \\
 \{x\} & \text{ if not} ,\end{array} \right. \\
  \mathcal{P}^{-} (x,y,\varphi) & =&
\left\{ \begin{array}{ll}
\{ z \in B_R(y) :  \varphi (z) \le  \varphi (y) \} &  \text{ if } D\varphi(y) \neq 0 \text{ and }   \kappa_* [y,\varphi] <  0 \\
  \{ x \} & \text{ if not} .
\end{array} \right.
  \end{eqnarray*}

Let us also define the following operators
\begin{eqnarray*}
\RE [\phi](t,x)  &=&  \sup_{ (y,\vp) \in \mathcal{C}^+ (x)
} \inf_{z \in {\mathcal{P}}^+ (x,y,\vp) }
\phi (t+ T_P (y,\vp),z) \, , \\
\Re [\phi](t,x) &=&  \inf_{ (y, \vp) \in \mathcal{C}^-
(x)}
        \sup_{z \in {\mathcal{P}}^- (x,y,\vp) } \phi (t+ T_C (y,\vp),z) \, .
\end{eqnarray*}
The reader can notice that
\begin{eqnarray}\label{SR}
\Sc^\eps [\phi](t,x) &=& \RE [\Re [\phi]] (t,x) \, , \\
\nonumber \Re[\phi] (t,x) &=& -\RE [-\phi] (t,x) \,
\end{eqnarray}
In order to get the second equality, we need to remark that
\begin{eqnarray*}
(z,\varphi) \in \mathcal{C}^+ (y) &\Leftrightarrow& (z, -\varphi)
 \in \mathcal{C}^- (y) \, , \\
z \in {\mathcal{P}}^+ (y,\varphi) &\Leftrightarrow& z \in {\mathcal{P}}^- (y,-\varphi) \, ,\\
T_P (z,\varphi) &=& T_C (z,-\varphi) \, .
\end{eqnarray*}
Moreover, the operator $\RE$ is monotone and commutes with constants:
\begin{eqnarray}
\label{monotonieR}
  \phi_1 \le \phi_2 &\Rightarrow& \RE[\phi_1 ]\le \RE[\phi_2]
\\
\label{rpc}
 \RE[\phi+ c]&=& \RE[\phi ] + c
\end{eqnarray}
for all $c \in \R$.
The proof of Proposition~\ref{lem:dislo-semi} relies on four consistency lemmas.
 Before stating  them, let us point out that we will write
$\kappa[\cdot ]_+ $ for the positive part of $\kappa[\cdot]$ and
$\kappa[\cdot]_-$ for the negative part (both being nonnegative).
\begin{lem}[Estimate from below for $\RE$] \label{lem:cdr-1} Consider
  a $C^2$ function $\phi : (0,T] \times \R^N \to \R$ and $(t_0,x_0)
  \in (0,T) \times \R^N$.  There exists a function $o (1)$ depending
  on $\phi$ and $(\eps,r)$ such that $o (1) \to 0$ as $(\eps, r)\to 0$
  and such that for all $(t,x) \in B_r (t_0,x_0)$ there exists
  $(y,\varphi) \in \mathcal{C}^+ (x)$ such that
  \begin{itemize}
  \item
     if $D\phi (t_0,x_0) \neq 0$,
\begin{equation}\label{eq:consis-red-geq1}
 \RE[\phi](t,x) - \phi (t,x) \ge T_P (y,\varphi)  \bigg(\partial_t \phi (t_0,x_0)
+ \kappa_* [x_0, \phi (t_0,\cdot)]_+  | D \phi (t_0,x_0)| + o (1)\bigg),
\end{equation}
 \begin{equation}
 \label{cdp1}
 |D\varphi(y)|\ge \hal |D\phi(t_0,x_0)|
\end{equation}
and
\begin{equation}\label{eq:consis-red-geq2}
 \k_*[x_0,\phi (t_0,\cdot)] + o (1)\le \k_* [y,\varphi] \le
\k^* [y,\varphi] \le \k^*[x_0,\phi (t_0,\cdot)] + o (1);
\end{equation}
\item
if $D\phi (t_0,x_0) = 0$,  \eqref{eq:consis-red-geq1} still holds
true with the convention
$$\kappa_* [x_0, \phi (t_0,\cdot)]_+  | D \phi (t_0,x_0)|=0$$
and $T_P (y,\varphi)=\eps^2$.\end{itemize}
\end{lem}
\begin{lem}[Estimate from above for $\RE$] \label{lem:cdr-2} Consider
  a $C^2$ function $\phi : (0,T] \times \R^N \to \R$ and $(t_0,x_0)
  \in (0,T) \times \R^N$. There exists a function $o (1)$ such that $o
  (1) \to 0$ as $(\eps, r) \to 0$ and such that for all $(t,x) \in B_r
  (t_0,x_0)$
\begin{itemize}
\item
either
\begin{equation}\label{eq:consis-dislo-c1}
\RE [\phi] (t,x) - \phi (t,x) \le \eps^2 (\partial_t \phi(t_0,x_0) + o (1) ) \, ;
\end{equation}
\item or $D\phi(t_0,x_0)= 0 $ and
\begin{equation}\label{consis-dislo-c1b}
\RE[\phi] (t,x)-\phi(t,x)\le T_P(y, \vp) (\p_t \phi(t_0,x_0) +o(1)) \end{equation} for some $(y, \vp) \in \mathcal{C}^+(x);$
\item or $D\phi(t_0,x_0) \neq 0$ and $\k^*[x_0,\phi(t_0,\cdot)] \ge 0$ and \begin{multline}\label{eq:consis-dislo-c2}
  \RE [\phi] (t,x) - \phi (t,x) \le  T_P (y,\varphi) (\partial_t \phi (t_0,x_0)
+ \kappa^*[x_0, \phi (t_0,\cdot)] |D\phi (t_0,x_0)| + o (1))
\end{multline}
for some $(y,\varphi) \in \mathcal{C}^+(x)$ such that $T_P (y,\varphi) = \min\( \eps / \k^* [y,\vp], \ep^{\hal}\)$
with\begin{equation}\label{328b}
0<\k^*[y,\vp] \le \k^*[x_0,\phi(t_0,\cdot)] + o(1).
\end{equation}
\end{itemize}
\end{lem}
By using the fact that $\Re[\phi]=- \RE[-\phi]$ and exchanging the roles of $+$ and $-$, we then can deduce from
the two previous lemmas the two following ones.
\begin{lem}[Estimate from above for $\Re$] \label{lem:cdr-1bis}
  Consider a $C^2$ function $\phi : (0,T] \times \R^N \to \R$ and
  $(t_0,x_0) \in (0,T) \times \R^N$.  There exists a function $o (1)$
  depending on $\phi$ and $(\eps,r)$ such that $o (1) \to 0$ as
  $(\eps,r)\to 0$ and such that for all $(t,x) \in B_r (t_0,x_0)$
  there exists $(y,\varphi) \in \mathcal{C}^- (x)$ such that
  \begin{itemize}
  \item  if $D\phi (t_0,x_0) \neq 0$,
\begin{equation}\label{eq:consis-red-geq1bis}
\Re [\phi](t,x) - \phi (t,x) \le T_C (y,\varphi)  \bigg(\partial_t \phi (t_0,x_0)
- \kappa^* [x_0, \phi (t_0,\cdot)]_-  | D \phi (t_0,x_0)| + o (1)\bigg)
\end{equation}
\begin{equation}\label{cdp2}
|D\vp(y)|\ge \hal |D \phi (t_0,x_0)|
\end{equation}
and
\begin{equation}\label{eq:consis-red-geq2bis}
\k_*[x_0,\phi (t_0,\cdot)] + o (1)\le  \k_* [y,\varphi] \le
\k^* [y, \vp]\le \k^* [x_0,\phi (t_0,\cdot)] + o (1);
\end{equation}
\item
if $D\phi (t_0,x_0) = 0$, then \eqref{eq:consis-red-geq1bis} still holds
true with the convention $$\kappa^* [x_0, \phi (t_0,\cdot)]_-  | D \phi (t_0,x_0)|=0$$
and $T_C (y,\varphi) = \eps^2$.
\end{itemize}\end{lem}
\begin{lem}[Estimate from below for $\Re $] \label{lem:cdr-2bis}
  Consider a $C^2$ function $\phi : (0,T] \times \R^N \to \R$ and
  $(t_0,x_0) \in (0,T) \times \R^N$.  There exists a function $o
  (1)$ such that $o (1) \to 0$ as $(\eps,r)\to 0$ and such that for
  all $(t,x) \in B_r (t_0,x_0)$
\begin{itemize}
\item
either
\begin{equation}\label{eq:consis-dislo-c1bis}
\Re [\phi] (t,x) - \phi (t,x) \ge \eps^2 (\partial_t \phi(t_0,x_0) + o (1)) ;
\end{equation}
\item
or $D\phi(t_0,x_0)=0$ and
\begin{equation}
\label{dislo-c1b}
\Re[\phi](t,x)-\phi(t,x) \ge T_C(y, \vp) (\p_t \phi(t_0,x_0)+o(1))\end{equation}
for some $(y,\vp) \in \mathcal{C}^-(x);$
\item
or $D\phi(t_0,x_0) \neq  0$ and $\k_*[x_0,\phi(t_0,\cdot)] \le 0$ and
\begin{multline}\label{eq:consis-dislo-c2bis}
\Re [\phi] (t,x) - \phi (t,x) \ge  T_C (y,\varphi) (\partial_t \phi (t_0,x_0)
+ \kappa_*[x_0, \phi (t_0,\cdot)] |D\phi (t_0,x_0)| + o(1))
\end{multline}
for some $(y,\varphi) \in \mathcal{C}^-(x)$ such that $T_C (y,\varphi) =  \min \(\eps / |\k_* [y,\vp]|, \ep^{\hal}\)$
with
$$
0>\k_*[y,\vp] \ge \k_*[x_0,\phi(t_0,\cdot)] + o(1).
$$
\end{itemize}
\end{lem}
The proofs of Lemmas~\ref{lem:cdr-1} and
\ref{lem:cdr-2} are postponed. We now explain how to
derive Proposition~\ref{lem:dislo-semi}.
\begin{proof}[Proof of Proposition~\ref{lem:dislo-semi}]
  We only prove that $\overline{u}$ is a sub-solution of
  \eqref{eq:dislo} since a symmetric argument can be used to prove that
  $\underline{u}$ is a super-solution.  In order to do so, we consider
  a $(t_0,x_0) \in (0,T) \times \R^N$ and a $\phi \in C^2$ such that
  $\underline{u} - \phi$
  attains a strict maximum at $(t_0,x_0)$ in $(0,T) \times B_{R+1} (x_0)$. We want to prove
  that $-\partial_t \phi (t_0,x_0) - \k^* [x_0,\phi (t_0,\cdot)] | D
  \phi (t_0,x_0)| \le 0$ if $D\phi (t_0,x_0) \neq 0$ and $- \p_t \phi(t_0,x_0) \le 0$ if $D\phi(t_0,x_0)=0$.

  We know that there exists a sequence $(t_{\eps_n},x_{\eps_n})$ such
  that $u^{\eps_n} - \phi$ attains a maximum in $(0,T) \times B_{R}
  (x_{\eps_n})$ at $(t_{\eps_n},x_{\eps_n})$.  For simplicity, we
  simply write $(t,x)$ for $(t_{\eps_n},x_{\eps_n})$ and $\eps$ for
  $\eps_n$. With the same argument as in the proof of Propositions
  \ref{pro:pide} and \ref{lem:eik-sup}, the dynamic programming
  principle \eqref{pdp:dislo-bis} and the monotonicity of $\Sc_\ep$
  imply that
\begin{equation}
\label{posi}
\phi (t,x) \le \Sc^\eps [\phi] (t,x).
\end{equation}
We now estimate $\Sc^\eps [\phi] (t,x)$ from above. We distinguish cases.

\paragraph{Case $D\phi(t_0,x_0) \neq 0$ and $\kappa^* [x_0, \phi(t_0,
  \cdot)]>0$.} Lemma~\ref{lem:cdr-1bis} yields an $(y,\vp) \in
\mathcal{C}^-(x)$, with $\k_*[y,\vp]>0$ (by
\eqref{eq:consis-red-geq2bis}) for $\ep$ small enough, thus
$T_C(y,\vp)=\ep^2$, so that
  $$\Re[\phi](t,x)-\phi(t,x)\le \ep^2 \( \p_t \phi(t_0,x_0)+o(1)\)$$ and
  using properties \eqref{monotonieR}--\eqref{rpc}, we find
\begin{eqnarray*}
\Sc^\eps [\phi] (t,x) & = & \RE [\Re [\phi]](t,x) \\
& \le & \RE[\phi] (t,x) + \eps^2 (\partial_t \phi (t_0,x_0) + o(1)).
\end{eqnarray*}
We  now use Lemma~\ref{lem:cdr-2}. If \eqref{eq:consis-dislo-c1}
holds true, then (with \eqref{posi})
$$
0 \le \Sc^\eps [\phi] (t,x) - \phi (t,x) \le 2\eps^2 (\partial_t \phi(t_0,x_0)+o(1))
$$
and we conclude that $\partial_t \phi(t_0,x_0) \ge 0$. The result follows
easily in this subcase. The subcase where \eqref{consis-dislo-c1b} holds  works similarly.

If now \eqref{eq:consis-dislo-c2} holds true, we get
\begin{multline*}
0 \le\Sc_\eps [\phi] (t,x) - \phi (t,x) \le T_P (y,\vp) (\partial_t
\phi (t_0,x_0) + \k^* [x_0,\phi (t_0,\cdot)] |D \phi (t_0,x_0)|+ o (1))\\
+ \ep^2 \( \p_t \phi(t_0,x_0)+o(1)\)
.\end{multline*} Since $T_P(y,\vp)=\min\( \ep/\k_*[y,\vp], \ep^{\hal}\)$, this can be written as
$$0 \le\Sc_\eps [\phi] (t,x) - \phi (t,x) \le T_P(y,\vp) (\partial_t
\phi (t_0,x_0) + \k^* [x_0,\phi (t_0,\cdot)] |D \phi (t_0,x_0)|+ o (1))$$
and dividing by $T_P(y,\vp)$ and letting $\ep \to 0$, the desired inequality follows.

\paragraph{Case  $D\phi(t_0,x_0) \neq 0$ and $\kappa^* [x_0, \phi(t_0, \cdot)]< 0$.}
We apply first Lemma \ref{lem:cdr-1bis} and find
\begin{equation*}
 \Re [\phi](t,x) - \phi (t,x) \\\le
T_C (y,\varphi) (\partial_t \phi (t_0,x_0)+ \k^*[x_0,\phi(t_0,x_0)] |D\phi(t_0,x_0)|
+ o(1)).
\end{equation*}
We note that from \eqref{cdp2} we have $D\vp(y) \neq 0$.
Now we cannot have $\k_*[y,\varphi] \ge 0$, otherwise a
contradiction would follow from \eqref{eq:consis-red-geq2bis} and our
assumption $\k^* [x_0, \phi(t_0, \cdot)]< 0$. We deduce that the case
$T_C(y, \vp)= \ep^2 $ cannot happen and we must have $\k_*
[y,\varphi]<0$ and $ \eps^2 = o(T_C(y,\varphi)).  $
With this piece of information at hand, we can write, as previously
\begin{multline*}
0 \le \Sc_\eps [\phi](t,x) - \phi (t,x) \le \RE [\phi](t,x) - \phi(t,x)
\\+ T_C (y, \varphi) (\partial_t \phi (t_0,x_0)
+ \k^*[x_0,\phi(t_0,\cdot)]|D\phi(t_0,x_0)| + o (1)).
\end{multline*}
On the other hand, Lemma \ref{lem:cdr-2} yields (we can only be in the
first situation of the lemma) $$\RE[\phi](t,x) \le \phi (t,x) +
O(\eps^2)= \phi(t,x)+o(T_C(y,\vp)),$$ and we can write, dividing by
$T_C (y, \varphi)$,
$$
0 \le \partial_t \phi (t_0,x_0)
+ \k^*[x_0,\phi(t_0,\cdot)]|D\phi(t_0,x_0)| + o(1)
$$
and the desired inequality is thus obtained in this case too.

\paragraph{Case $D\phi(t_0,x_0) \neq 0$ and $\kappa^* [x_0, \phi(t_0, \cdot)]=0$.}
Once again, we  first apply Lemma~\ref{lem:cdr-1bis} and we obtain
$$
0 \le \Sc_\eps [\phi](t,x) - \phi (t,x) \le\RE [\phi](t,x) - \phi(t,x)
+ T_C (y,\varphi) (\partial_t \phi (t_0,x_0)
+ o(1)).
$$
Lemma~\ref{lem:cdr-2} implies that
$$
\RE[\phi](t,x) - \phi (t,x) \le M_\eps (\partial_t \phi(t_0,x_0)+o(1))
$$
with $M_\eps =\eps^2$ or $M_\eps = T_P(y, \vp)$. Hence we obtain
$$
0 \le (T_C (y,\varphi) + M_\eps) (\partial_t \phi (t_0,x_0) + o(1))
$$
and we obtain the desired inequality in this case too.

\paragraph{Case $D\phi(t_0,x_0) = 0$.} Then Lemmas~\ref{lem:cdr-1bis} and \ref{lem:cdr-2}
yield
$$
0 \le (\eps^2+T_P(y, \vp)) (\p_t \phi(t_0,x_0) + o(1))
$$
or $0 \le 2\ep^2(\p_t \phi(t_0,x_0) + o(1))$,
and the proof of the proposition is now complete.
\end{proof}
We now turn to the proofs of Lemmas~\ref{lem:cdr-1}
and \ref{lem:cdr-2}. As the reader shall see, we
follow along the lines of proofs of Lemmas~\ref{lem:consis-eik} and
\ref{lem:consis-eik2} used in the eikonal case.
\begin{proof}[Proof of Lemma~\ref{lem:cdr-1}]
We first assume that $D\phi(t_0,x_0) \neq 0$. So we can assume that,
for
$\eps$ small enough, $|D\phi (t,x)| \ge \theta_0 >0$.

Consider $y = x + \eps \frac{D\phi (t,x)}{|D\phi (t,x)|}$ and
$\vp (z) = \phi (t,z) - \alpha_\eps (z)$ with $\alpha_\eps : \R^N
\to [0,+\infty)$ smooth and
$$
\alpha_\eps (z) = \left\{ \begin{array}{ll} 0 & \text{ if } |z-x| \le \eps \\
\eps^{\frac14} & \text{ if } |z-x| \ge 2 \eps .
\end{array} \right.
$$
We also can write for $\eps$ small enough
$$
\varphi (y) = \phi (t,y) = \phi (t,x) + \eps |D \phi (t,x)| + O (\eps^2)
\ge \phi (t,x) = \varphi (x)
$$
(we used the fact that $|D\phi(t,x)| \ge \theta_0>0$).
This means that
$(y,\varphi) \in \mathcal{C}^+ (x)$ (at least for $\eps$
small enough). Remark also that \eqref{cdp1} holds. Hence
$$
\RE[\phi](t,x) \ge \inf_{z \in \mathcal{P}^+
(x,y,\vp)} \phi (t+ T_P (y,\vp), z)  .
$$
Since $(t,y )=(t_0,x_0) + o_\eps (1)$ and $\varphi \le
\phi(t,\cdot)$, it follows that
$$
\{ z|\vp( z) \ge \varphi (y)=
\phi(t, y) \}\subset\{ z|\phi(t,z)\ge \phi(t, y) \}
$$
and, using the monotonicity of $\k^*$ and its upper
semi-continuity (see \eqref{eq:monotonie++} and
\eqref{eq:usc-courbure}),  we conclude that
\begin{equation}\label{cx1}
\k^*[y, \varphi] \le \k^*[y, \phi(t,\cdot)]
\le \k^*[x_0, \phi(t_0,\cdot)] + o_\eps (1)
\end{equation}
for $\eps$ small enough; the other part of \eqref{eq:consis-red-geq2}
follows similarly by lower semi-continuity of $\k_*$
\eqref{eq:lsc-courbure}.  \eqref{eq:consis-red-geq2} is thus proved.
Moreover \eqref{cx1} implies $\k^*[y, \vp] <
\frac{1}{\sqrt{\ep}}$ hence either $T_P (y,{\vp})=\eps^2$
or $T_P(y,{\vp}) =\frac{\eps}{\k^*[y,\varphi]_+} $
or $T_P(y,{\vp}) = \eps^{\frac12}$. We treat these cases
separately.

\paragraph{Case $T_P(y,{\vp})=\eps^2$.} We know that in this case
we have $\k^* [y,{\vp}] \le 0$ and
$\mathcal{P}^+(x,y,{\vp})=\{x\}$. Hence,
\begin{eqnarray*}
\RE[\phi](t,x) - \phi (t,x) &\ge& \phi (t+\eps^2,x) - \phi(t,x) + O (\eps^3)  \\
& \ge & \eps^2( \partial_t \phi (t,x) + o_\eps (1)).
\end{eqnarray*} Since $\k^* [y,{\vp}] \le 0$ we deduce from \eqref{eq:consis-red-geq2} $
\k_*[x_0,\phi(t_0,\cdot)] \le 0.
$
hence \eqref{eq:consis-red-geq1} is proved in this case.
\medskip

In the two remaining cases, we have $T_P(y,{\vp}) \le \eps^{\frac12}$ and
$\mathcal{P}^+ (x,y,{\vp} ) \ni y$. These two facts imply the following
inequality
\begin{equation}\label{eq:localise}
 \RE[\phi](t,x) \ge \inf_{z \in \mathcal{P}^+ (x,y,\varphi) \cap
  B_{2\ep}(x)} \phi (t+ T_P (y,{\vp}), z).
\end{equation} i.e. the fact that the infimum can only be achieved in $B_{2\ep}(x)$.
To see this, we simply write for $z \in \mathcal{P}^+
(x,y,\varphi)$ such that $z \notin B_{2\ep}(x)$,
\begin{eqnarray*}
 \phi (t+ T_P (y,{\vp}), z) &=& \phi (t,z) + O (T_P(y,{\vp})) \\
& \ge & \phi (t,y) + \eps^{\frac14} + O (\eps^{\frac12}) \\
& \ge & \phi (t+ T_P (y,{\vp}),y)
+ \eps^{\frac14} 
+ O (\eps^{\frac12}) \\
& > & \inf_{z \in \mathcal{P}^+ (x,y,\varphi) \cap B_{2\ep}(x)}
\phi (t+ T_P (y,{\vp}), z)
\end{eqnarray*}
and \eqref{eq:localise} follows.

\paragraph{Case $T_P(y,{\vp}) = \eps^{\frac12}$.} By definition of $C_\ep$, this   happens if  $\k^*[y,\varphi] \le \eps^{\frac12}$ and from \eqref{eq:consis-red-geq2} it follows that
$
\k_*[x_0, \phi(t_0,\cdot)] \le 0.
$
 For $z \in
\mathcal{P}^+ (x,y,\varphi) \cap B_{2\ep}(x)$, we have
$$
\phi (t+ T_P (y,{\vp}), z) - \phi (t,x) \ge \phi (t+\eps^{\frac12},z)- \phi (t,z) +
\phi (t,z) - \phi(t,x) + O (\eps^3)$$
But $\vp(z) \ge \vp(y) $ since $z \in \mathcal{P}^+(x,y,\vp)$, hence
$\phi(t,z) - \alpha_\ep(z) \ge \phi(t,y) $ so  replacing in the above and using $\alpha_\ep \ge 0$ we are led to
\begin{eqnarray*}
\phi (t+ T_P (y,{\vp}), z) - \phi (t,x) & \ge &
\phi (t+\eps^{\frac12},z)- \phi (t,z) +
\phi (t,y) - \phi(t,x) + O (\eps^3)
\\
& \ge & \eps^{\frac12} \partial_t \phi (t,x)   + \eps |D \phi (t,x)| + o_\eps (\ep) \\
& \ge & \eps^{\frac12} (\partial_t \phi (t,x)  + o_\eps (1)) \\
& \ge & \eps^{\frac12} (\partial_t \phi (t,x) + \k_*[x_0,\phi(t_0,\cdot)]_+ |D \phi (t_0,x_0)| + o_\eps (1))
\end{eqnarray*}
and we get \eqref{eq:consis-red-geq1} in this case too.

\paragraph{Case $T_P(y,{\vp})=\frac{\eps}{\k^*[y, \varphi]_+} $.}
 Observe that from \eqref{cx1}, $\k^*(y,\vp)$ is bounded above hence $T_P(y,\vp)$ bounded below by $c\ep$.  As above, we may  write, recalling the choice of $y= x+ \ep \frac{D\phi(t,x)}{|D\phi(t,x)|}$
\begin{eqnarray*}
  \RE [\phi](t,x) - \phi (t,x) & \ge
  & \inf_{z \in \mathcal{P}^+ (x,y,\varphi) \cap B_{2\ep}(x)}
\phi (t+ T_P (y,{\vp}), z) - \phi (t,x)\\
  & \ge & \inf_{z \in \mathcal{P}^+ (x,y,\varphi) \cap B_{2\ep}(x)}
\phi (t+ T_P (y,{\vp}), z) - \phi (t,z)
  + \phi (t,z) - \phi (t,x)\\
  & \ge &  T_P (y,\vp) (\partial_t \phi (t,x)
  + o_\eps (1)) + \phi (t,y) - \phi (t,x) \\
  & \ge &  T_P (y,\vp) (\partial_t \phi (t,x)
  + o_\eps (1)) + \eps |D \phi (t,x)| + O (\eps^2) \\
  & \ge &  T_P (y,\vp) \bigg( \partial_t \phi (t,x)
+ \frac\eps{T_P (y,\vp) }|D\phi (t,x)|
 +o_\eps (1) \bigg) \\
 & \ge& T_P (y,\vp) \bigg( \partial_t \phi (t,x)  +  \k^* [y,\varphi]_+
 |D\phi (t,x)| +o_\eps (1) \bigg) \\
& \ge & T_P (y,\vp) \bigg( \partial_t \phi (t,x)  +  \k_* [x_0,\phi(t_0,\cdot)]_+
 |D\phi (t_0,x_0)| +o_\eps (1) \bigg)
\end{eqnarray*} where the last inequality follows from \eqref{eq:consis-red-geq2} ;
and we get \eqref{eq:consis-red-geq1} in all cases.
\medskip

Assume now that $D\phi (t_0,x_0) = 0$. Then choose $y=x$
and $\varphi (z)= -\alpha_\eps (z)$. This is admissible and $D\varphi(y) =0$ and $T_P(y, \vp)=\ep^2$ and the conclusion follows easily.
\end{proof}
We now turn to the proof of Lemma~\ref{lem:cdr-2}.
\begin{proof}[Proof of Lemma~\ref{lem:cdr-2}]
We recall that
\begin{equation}\label{rs1}
\RE[\phi](t,x) =  \sup_{(y,\vp) \in \mathcal{C}^+ (x)} \inf_{z \in \mathcal{P}^+(x,y,\vp) }  \phi( t+ T_P(y,\vp),
z)
\end{equation}
In view of the definition of  $\mathcal{C}^+(x)$ and $\mathcal{P}^+ (x,y,\vp)$,
we can write more precisely
\begin{multline}\label{rs0}
\RE[\phi](t,x) \\= \sup_{\vp \in C^2 (\R^N)} \max \bigg(
  \sup_{\stackrel{y \in B_\eps (x) :\vp (y) \ge \vp(x)}{ D \vp (y) \neq 0, \  \k^* [y,\vp] >0, }}
\inf_{z \in \mathcal{P}^+(x,y,\vp) } \phi( t + T_P (y,\vp),z) , \phi (t+\eps^2,x) \bigg).
\end{multline}
Let $\vp$ be a fixed $C^2$ test function.

\paragraph{Case 1.} Assume first that  the max above is
$  \phi (t+\eps^2,x)$.
Then we easily obtain \eqref{eq:consis-dislo-c1} as desired.

\paragraph{Case 2.}
We then turn to the situation where
\begin{equation}\label{eq:avant-cs}
 \sup_{\stackrel{y \in B_\eps (x):\vp (y) \ge \vp(x)}{D \vp (y) \neq 0, \k^* [y,\vp] >0}}
\inf_{z \in \mathcal{P}^+(x,y,\vp) }  \phi( t + T_P (y,\vp),z)  > \phi (t+\eps^2,x)  .
\end{equation}

We need to prove that \eqref{consis-dislo-c1b} or  \eqref{eq:consis-dislo-c2} holds true in this
case. So let $y \in B_\ep(x)$ be such that $\vp(y) \ge \vp(x) $,
$D\vp(y) \neq 0$ and $\k^*[y,\vp] >0$,  and such that
\begin{equation}\label{cx3}
  \inf_{z \in \mathcal{P}^+(x,y,\vp) }  \phi( t + T_P (y,\vp),z)  > \phi (t+\eps^2,x) - \ep^3.\end{equation}

For any $z\in \mathcal{P}^+(x,y,\vp) = \{ z \in B_R(y) : \vp (z) \ge
\vp (y)\}$ we may write
\begin{eqnarray*}
\phi(t^+, z)-\phi (t,x) &=& \phi(t^+, z)- \phi(t^+, y)
+ \phi (t^+,y) - \phi (t,x)\\
& =& \phi(t^+, z)- \phi(t^+, y) \\
&& + T_P(y,\vp)(\p_t \phi(t_0,x_0) + o_\eps (1))+ D\phi(t^+,y) \cdot (y-x) +O(\ep^2),
\end{eqnarray*}
and using the fact that $|y-x|\le \ep$ and $O(\eps^2)= o (t^+-t)$
since $T_P(y,\vp) \ge \ep^{3/2}$ in this case; we obtain
\begin{multline} \label{rs2} \phi(t^+, z)- \phi(t,x)\le \phi(t^+, z)-
  \phi(t^+,y)+ \ep |D\phi(t^+,y)|\\+ T_P(y, \vp) (\p_t
  \phi(t_0,x_0)+o_\eps (1)) .
\end{multline}
We now evaluate $\phi(t^+, z)- \phi(t^+, y)$. In view of \eqref{rs0},
\eqref{eq:avant-cs} and \eqref{rs2}, the following lemma permits to conclude.
\begin{lem}\label{lclaim} For any $(y,\vp) \in \mathcal{C}^+(x)$ with
  $D\vp(y)\neq 0$, $\k^*[y,\vp]>0$, such that \eqref{cx3} holds, we
  have
\begin{multline}\label{rs3}
\inf_{z \in \mathcal{P}^+(x,y,\vp) } \phi (t^+,z) - \phi (t^+,y) \le
T_P (y,\varphi) (\k^* [x_0,\phi (t_0,\cdot)]_+ |D\phi (t_0,x_0)| + o_\eps (1) ) \\
-\eps | D \phi (t^+,y)| .
\end{multline}
Moreover, if $D\phi(t_0,x_0)\neq 0$, we have
$$
0 <\k^*[y,\vp]\le \k^*[x_0, \phi(t_0, \cdot)] \text{ and } T_P(y,\vp)= \min \( \ep/\k^*(y,\vp), \ep^{\hal}\).
$$

\end{lem}

There now remains to give the proof of Lemma~\ref{lclaim}.
We start by
 \begin{lem}\label{lem:comparaison-courbure}
If  $D\phi(t_0,x_0) \neq 0$, then for any $(y,\vp)$  as in the above lemma, we have
$$
0 <   \k^*[y,\vp] \le \k^* [x_0, \phi(t_0 , \cdot)].
$$
\end{lem}
\begin{proof}[Proof of Lemma~\ref{lem:comparaison-courbure}]
  We are in the case where \eqref{cx3} holds true, hence for all $z \in \{
  z\in B_R(y) : \vp (z) \ge \vp (y) \}$, we have $\phi( t + T_P
  (y,\vp),z) \ge \phi (t+\eps^2,x) -\ep^3$. This implies
\begin{equation}\label{inclset}
\{ z \in B_R(0) : \vp (y+z) \ge \vp (y) \} \subset \{ z \in B_R(0) :
\phi(t_0,x_0+z) \ge \phi(t_0,x_0) \}.
\end{equation}
Indeed, if $z \in B_R(0)$ is such that $\phi (t_0,x_0+z) < \phi
(t_0,x_0)$, then for $\ep $ small enough and $(t,x)-(t_0,x_0)$ small
enough, $\phi (t+T_P (y,\vp),y+z) < \phi (t+\eps^2,x) -\ep^3$ and in view of the above this
implies $\vp (y+z) < \vp (y)$.  By properties of the non-local
curvature, \eqref{inclset} implies \eqref{cx2} and the proof of
the lemma is complete.
\end{proof}
We can now complete the proof of Lemma \ref{lclaim} by
\begin{proof}[Proof of \eqref{rs3}]
Remark first that Lemma~\ref{lem:comparaison-courbure} implies in
particular that
 \begin{equation}\label{cx2}
 \k^*[y,\vp] |D \phi (t_0,x_0)| \le \k^* [x_0, \phi(t_0 , \cdot)]  |D \phi (t_0,x_0)|
\end{equation}
since this inequality is trivial when $D \phi (t_0,x_0)=0$.

We now argue by contradiction.  We
  thus assume that there exists $\eta >0$, $\eps_n \to 0$, $(t_n,x_n)
  \to (t_0,x_0)$ and $(y_n,\vp_n)\in \mathcal{C}^+(x_n) $ such that
  $D\vp_n(y_n) \neq 0$, $\k^* [y_n,\vp] >0$ and \eqref{cx3} holds,  and for all $z \in
  B_R(y_n) \cap \{ \vp_n \ge \vp_n (y_n) \}$, we have
\begin{equation}\label{claim:depart}
 \phi (t^+_n,z) - \phi (t^+_n,y_n) \ge
T_P (y_n,\varphi_n) (\k^* [x_0,\phi (t_0,\cdot)]_+ |D\phi (t_0,x_0)| + \eta)
-\eps_n | D \phi (t^+_n,y_n)| .
\end{equation}
It then follows from \eqref{cx2} that for $n$ large enough
\begin{equation}\label{338bis} \phi (t^+_n,z) - \phi (t^+_n,y_n) \ge
T_P(y_n,\vp_n) \( \k^*[y_n,\vp_n] |D\phi(t_n^+,y_n)|+\frac{\eta}2\) -\eps_n | D \phi (t_n^+,y_n)| .
\end{equation}
Assume first that there exists a subsequence such that $T_P (y_n,\vp_n) = \eps_n^{\frac12}$.
In this case, for all $z \in B_R(y_n) \cap  \{ \vp_n \ge \vp_n (y_n) \}$, we find
$$
 \phi (t^+_n,z) - \phi (t^+_n,y_n) \ge \frac{\eta}{2} \eps_n^{\frac12} + O (\eps_n) > 0
$$
for $n$ large enough. We obtain a contradiction by taking $z=y_n$.

Either we have $T_P(y_n,\vp_n)=\ep^{3/2}$ or $T_P (y_n,\varphi_n) = \eps_n / \k^*[y_n,\vp_n]$. In both situations we have
 $T_P (y_n,\varphi_n) \ge \eps_n / \k^*[y_n,\vp_n]$.
Choosing $z=y_n$ in \eqref{338bis} yields
$$
0 \ge T_P(y_n,\vp_n) \frac\eta{2}
$$
which is a contradiction.
\end{proof}
\end{proof}
We next prove that the terminal condition is satisfied at the limit.
\begin{proof}[Proof of Lemma~\ref{lem:dislo-ci}]
The proof consists in proving the following estimate
\begin{equation}\label{eq:tc1}
|u^\eps (t,x) - u_T (x) | \le C (T-t)
\end{equation}
for $t<T$ and $x \in \R^N$ with
$$
C = \sup_{x \in \R^N} \max (|\k_*[x,u_T]||Du_T (x)|+1,
|\k^*[x,u_T]||Du_T (x)|+1 ).
$$
We remark that \eqref{eq:tc1} is a consequence of the following lemma.
\begin{lem}\label{lem:tc-reduced}
Consider $k \in \mathbb{N} \cap (0,\eps^{-2} T)$.
If,
\begin{equation}\label{eq:tc1-bis}
\forall (t,x) \in (T-k\eps^2,T) \times \R^N, \quad |u^\eps (t,x)-u_T(x)| \le C (T-t),
\end{equation}
then
\begin{equation}\label{eq:tc2}
\forall (t,x) \in (T-(k+1)\eps^2,T) \times \R^N, \quad  |u^\eps (t,x)-u_T(x)| \le C (T-t).
\end{equation}
\end{lem}
\end{proof}

It remains to prove Lemma~\ref{lem:tc-reduced}.
\begin{proof}[Proof of Lemma~\ref{lem:tc-reduced}]
We only prove that for all $(t,x) \in (T-(k+1)\eps^2,T) \times \R^N$, we have
$$
u^\eps (t,x)\ge u_T(x) -  C (T-t)
$$
and the reader can check that the proof of the reverse inequality is similar.

  It is enough to consider $t \in (T-(k+1)\eps^2, T-k \eps^2)$.  We
  recall that the dynamic programming principle can be written as
  follows
$$
u^\eps (t,x) = \mathcal{S}^\eps [u^\eps] (t, x) = \RE [\Re [u^\eps]] (t,x).
$$

Thanks to Lemma~\ref{lem:cdr-2bis} and \eqref{eq:tc1-bis}, we know
that there exists $(y,\vp) \in \mathcal{C}^+ (x)$ such that
\begin{eqnarray*}
\Re[u^\eps ](t,x) &\ge &u_T (x) - C (T-(t+T_C (y,\vp)))
-  T_C (y,\vp) ( \k_* [x,u_T] |Du_T(x)| + 1) \\
&\ge &u_T (x) - C(T-t).
\end{eqnarray*}

We now use that $\RE$ is monotone and commutes with constants
(see \eqref{monotonieR} and \eqref{rpc}) in order to write
$$
u^\eps (t,x) \ge \RE [u_T](x) - C (T-t).
$$

We remark next that $(x,u_T)\in \mathcal{C}^+(x)$ and we  write
$$
 u^\eps (t,x) \ge \inf_{z \in \mathcal{P}^+ (x,x,u_T)} u_T( z) - C (T-t).
$$
We distinguish cases.

If $Du_T(x) \neq 0$ and $\k^*[x,u_T] >0$, then we have the desired
inequality; indeed,
\begin{eqnarray*}
 u^\eps (t,x) &\ge& \inf_{z : u_T(z) \ge u_T (x)} u_T(z) - C (T-t) \\
& \ge & u_T (x) - C(T-t).
\end{eqnarray*}

If now $Du_T (x)=0$ or $\k^*[x,u_T] \le 0$, then we also have
$$
u^\eps(t,x) \ge u_T (x) - C (T-t).
$$
The proof of Lemma~\ref{lem:tc-reduced} is now complete.
\end{proof}

\end{document}